\documentclass[12pt]{article}
\usepackage{amsmath, amsfonts, amsthm}
\usepackage{amssymb,graphics,epsfig}
\usepackage{enumerate}
\usepackage[all,color]{xy}    
\usepackage{color}         

\numberwithin{equation}{section}

\usepackage{geometry}
\newtheorem{defn}{Definition}
\newtheorem{theorem}{Theorem}
\newtheorem{prop}{Proposition}
\newtheorem{lemma}{Lemma}

\newtheorem{remark}{Remark}

\DeclareMathOperator{\diver}{\operatorname{div}}

\newcommand{\Black}{\color{black}}

\let\ds\displaystyle

\newcommand{\R}{ {\mathbb R} }

\newcommand{\N}{ {\mathbb N} }

\newcommand{\E}{ {\mathbb E} }
\newcommand{\PP}{ {\mathbb P} }

\newcommand{\ep}{{\varepsilon}}

\newcommand{\Supp}{\mathrm{Supp}}

\newcommand{\Winf}{W_\infty}
\newcommand{\tWinf}{{\tilde W}_\infty}

\newcommand{\dENinf}{| \nabla^N E|_\infty }
\newcommand{\inttau}{\int_{t-\tau}^t}
\newcommand{\Ialp}{I_\alpha}
\newcommand{\Jalp}{J_{\alpha+1}}
\newcommand{\bz}{{\bar z}}
\newcommand{\bx}{{\bar x}}
\newcommand{\bv}{{\bar v}}

\newcommand{\zz}{{\scriptscriptstyle Z}}

\newcommand{\cqfd}{{\unskip\kern 6pt\penalty 500
\raise -2pt\hbox{\vrule\vbox to 6pt{\hrule width 6pt
\vfill\hrule}\vrule}\par}}

\title{Particle approximation of 
Vlasov equations with singular forces: Propagation of chaos.
} 
\author{Maxime Hauray\footnote{
Universit{\'e} d'Aix-Marseille, CNRS, \'Ecole Centrale,  IMM, UMR 7373,  13453 Marseille, France}\;
 and Pierre-Emmanuel Jabin
\footnote{CSCAMM and Dpt of Mathematics, University of Maryland, College Park, MD 20742, USA}
}

\date{}

\begin{document}
\maketitle

{\bf Abstract.} 
We prove the mean field limit and the 
propagation of chaos for a system of particles interacting with a
singular interaction force of the type
 $1/|x|^\alpha$, with $\alpha <1$ in dimension $d \geq 3$.
We also provide results for forces with singularity up to $\alpha <
d-1$ but with a large  enough cut-off. This last result thus almost includes the case of Coulombian or gravitational interaction,
 but it also allows for a very small cut-off when the strength of the singularity $\alpha$ is 
larger but close to one.


\section{Introduction}

\paragraph{\bf The $N$ particle's system.}
  The starting point is the classical Newton dynamics for $N$
  point-particles. We denote by
 $X_i \in \R^d$ and $V_i \in \R^d$ the position
  and velocity of the $i$-th particle. 
  For convenience, we also use the notation $Z_i=(X_i,V_i)$ and 
   \mbox{$Z =(Z_1,\dots,Z_n)$}. 
  Assuming that particles interact two by two with the
 interaction force $F(x)$, one finds the classical
\begin{equation} \label{eq:ODE}
\left\{\begin{array}{l}
\ds \dot{X_i} = V_i, \\
\ds \dot{V_i}= E_N(X_i)=-\frac{1}{N}\displaystyle\sum_{j \neq i} F(X_i-X_j).
\end{array} \right.
\end{equation}
The ($N$-dependent)  initial conditions $Z^0$ are given. We use the so-called 
mean-field scaling which consists in keeping the total
mass (or charge) of order $1$ thus formally enabling us to pass to the
limit. This explains the $1/N$ factor in front
of the force terms, and implies corresponding rescaling in position,
velocity and time.

There are many examples of physical systems following
\eqref{eq:ODE}. The best known example concerns Coulombian or gravitational force
$F(x) =-\nabla \Phi(x)$, with 
$\Phi(x)= C/|x|^{d-2}$ with $C \in \R^\ast$, which serves as a guiding example and
reference.   This system then describes ions or electrons evolving in a plasma for $C>0$,  or  gravitational interactions for $C<0$. In the last case the system under study may be a galaxy, a smaller cluster of stars or much larger clusters of galaxies (and
thus particles can be ``stars'' or even ``galaxies''). 

For the sake of simplicity, we consider here only a basic form for the interaction. However the
same techniques would apply to more complex models, for instance with
several species (electrons and ions in a plasma), 3-particle (or more)
interactions, models where the force also depends  on the velocity as in
swarming models like Cucker-Smale \cite{CarCanBol}... 
Indeed a striking feature of our analysis is that it is valid for 
a force kernel $F$ not necessarily derived from a potential:
In fact it never requires any Hamiltonian structure.

\paragraph{The potential and force used in this article.}
Our first result  apply to interaction  forces that are smooth outside the origin and ``weakly'' singular near zero, in the sense that they satisfy
\begin{equation} \label{eq:Calpha}
(S^\alpha) \qquad \exists \, C>0,\quad\forall\,  x \in \R^d\backslash \{0\}, \quad |F(x)|  \leq \frac{C}{|x|^\alpha} \,
, \quad
|\nabla F(x) |  \leq \frac{C}{|x|^{\alpha+1}},
\end{equation}
for some $\alpha<1$.

We refer to this condition as the ``weakly'' singular case because under this, the potential (when it exists) is continuous and bounded near the origin. It  is reasonable to expect that the analysis is simpler in that case than with a singular potential. 

\medskip
The second type of potentials or forces that we are dealing with are more singular, satisfying the $(S^\alpha)$-condition with $\alpha < d-1$, but with a additional cut-off $\eta$ near the origin that will depends on $N$
\begin{equation} \label{eq:Ckappa} 
(S^\alpha_m) \qquad 
\begin{array}{ll}
i)  &   F \; \text{satisfy a } (S^\alpha)-\text{condition for some}\ \alpha<d-1, \\
ii) &   \forall \, |x| \geq N^{-m}, \,  F_N(x) = F(x), \\
iii)&    \forall \, |x| \leq N^{-m}, \,  |F_N(x) | \leq N^{m \alpha}.
\end{array} \end{equation}
We will refer to that case as the ``strongly'' singular case. Remark that the  interaction kernel $F$ in fact depends on the number of particles. This might seem strange from the
physical point of view but it is in fact very common in numerical
simulations in order to regularize the interactions. 

\Black
As the interaction force is singular, we first precise what we mean by solutions
to~\eqref{eq:ODE} in the following definition
\begin{defn}
A (global) solution to~\eqref{eq:ODE} with initial condition
$$
Z^0=(X_1^0,\;V_1^0,\ldots,X_N^0,\;V_N^0) \in \R^{2dN}
$$
(at time
$0$)  is a continuous trajectory
$Z(t)=\bigl(X_1(t),\;V_1(t),\ldots,X_N(t),\;V_N(t)\bigr)$  such that
\begin{equation} \label{eq:ODEint}
\forall t \in  \R^+, \; \forall i \le N, \quad 
\begin{cases} \displaystyle 
   X_i(t) = X_i^0 + \int_0^t V_i(s) \,ds  \\
\displaystyle   V_i(t) = V_i^0 + \frac1N \sum_{j \neq i} \int_0^t F(X_i(s) -
X_j(s))\,ds.
\end{cases} \end{equation} 
Local (in time) solutions are defined similarly.
\end{defn}
\Black
We always assume that such solutions to~\eqref{eq:ODE} exist, at least for almost all initial configurations of the particles and over any time interval $[0,\ T]$ under consideration. Of course, as we use singular interaction forces, this is not completely obvious,
but it holds under the assumption~\eqref{eq:Calpha}. This point is discussed at the end of the article in subsection~\ref{subsec:exis}, and we now focus on the problem raised by the limit $N \rightarrow + \infty$. 

Remark also that the uniqueness of such solutions is not important for our study.
Only the uniqueness of the solution to the limit equation is crucial for the mean-field limit and the propagation of chaos. 

%
\paragraph{\bf The Jeans-Vlasov equation.}
At first glance, the system \eqref{eq:ODE} might seem quite
reasonable.
However many problems arise when one tries to use it for practical
applications. 
In our case, the main issue is the number of particles,
{\em i.e.} the dimension of the system. 
For example a plasma or a galaxy usually contains a very large number of
``particles", typically from $10^9$ to $10^{25}$, which can make
solving \eqref{eq:ODE} numerically prohibitive. 

As usual in this kind of situation, one would like to replace the
discrete system \eqref{eq:ODE} by a ``continuous'' model. In our case
this model is posed in the space $\R^{2d}$, {\em i.e.} it involves
the distribution function $f(t,x,v)$ in time, position and velocity. 
The evolution of that  function
$f(t,x,v)$ is given by the Jeans-Vlasov equation (or collisionless Boltzmann equation)
\begin{equation}   \label{eq:vlasov}
\begin{cases}
& \ds \partial_t f + v \cdot \nabla_x f +  E(x) \cdot \nabla_v f = 0 \,, \\
& \ds E(x) = \int_{\R^d} \rho(t,y)\,F(x-y)\,dy, \\
& \ds \rho(t,x)= \int_{\R^d} f(t,x,v) \,dv,
\end{cases} \end{equation}
where here  $\rho$ is the spatial density and the initial density $f^0$
is given.

Our purpose in this article is to understand when and in which
sense, Eq. \eqref{eq:vlasov} can be seen as a limit of system
\eqref{eq:ODE}. This question is of importance for
theoretical reasons, to justify the validity of the Vlasov
equation for example. 
It also plays a role for
numerical simulation, and especially Particles in Cells methods which
introduce a large number of ``virtual'' particles 
(roughly around $10^6$ or $10^8$, to compare with the
real order mentioned above)
in order to  obtain a many particle system solvable numerically. The problem in that
case is to explain why it is possible to correctly approximate the
system by using much fewer particles. This would of course be ensured by the
convergence of \eqref{eq:ODE} to \eqref{eq:vlasov}.

We make use of uniqueness results for the solution of equation~\eqref{eq:vlasov}. 
The regularity theory for this equation is now well understood, even when the interaction $F$ is singular,
including the Coulombian case. The existence of weak solutions goes
back to \cite{Arse75,Dobr79}. Existence and uniqueness of global classical
solutions in dimension up to $3$ is proved in \cite{Pfaf}, \cite{Scha91} (see also \cite{Hor93}) and at
the same time in \cite{LioPer91}. 
Of course those results require some assumptions on the initial data
$f^0$: for instance compact support and boundedness in  \cite{Pfaf}. We will  state the precise result of existence and uniqueness we need in  Proposition~\ref{prop:WPvlasov} in Section~\ref{sec:uniq}.
%
\paragraph{\bf Formal derivation of Eq. \eqref{eq:vlasov} from \eqref{eq:ODE}.} 
One of the simplest way to understand formally how to derive
Eq. \eqref{eq:vlasov} is to introduce the empirical measure
\[ \label{eq:empdis}
\mu_N^{\zz}(t) = \frac{1}{N} \sum_{i=1}^N \delta_{X_i(t),V_i(t)}.
\]
In fact if  $Z(t) = \bigl(X_i(t),V_i(t)\bigr)_{1 \le i \le N}$ is a solution to \eqref{eq:ODE}, and if there is no
self-interaction: $F(0) =0$, then $\mu_N^\zz$ solves \eqref{eq:vlasov} in the
sense of distribution. Formally one may then expect that any limit of
$\mu_N^\zz$ still satisfies the
same equation. 

%
\paragraph{\bf The question of convergence and the  mean-field limit. 
}
The previous formal argument suggests a first way of rigorously
deriving the Vlasov equation \eqref{eq:vlasov}. Take a sequence of
initial conditions $Z^0_N$ (to be given for every number $N$ or a
sequence of such numbers) and assume that the corresponding empirical
measures at time $0$ converge (in the usual weak-$*$ topology for measures)
\[
\mu_N^\zz(0)\longrightarrow f^0(x,v).
\]
One would then try to prove that the empirical measures at later times
$\mu_N^\zz(t)$ weakly converge to a solution $f(t,x,v)$ to
\eqref{eq:vlasov} with initial data $f^0$. In other words, is the following
diagram commutative?
\[
\xymatrix{
      \mu_N^\zz(0) \ar@{~>}[r]^{\scriptscriptstyle \text{cvg}}
\ar[d]_{N part} &  f(0) \ar[d]^{VP} \\
      \mu_N^\zz(t) \ar@{~>}[r]^{\scriptscriptstyle \text{cvg {\bf ?}} }  & f(t)
    }
\]
We refer to the \emph{mean-field limit} for the question as to whether $\mu_N^\zz(t)$ converges to $f(t)$ for a given sequence of initial conditions $Z^0_N$ (or equivalently $\mu_N^0= \mu_N^\zz(0)$).  This is a purely deterministic problem. 
We give in Theorems~\ref{thm:deter} and~\ref{thm:cutoff} a quantified version of the convergence  
 $\mu_N(t)$ towards $f(t)$, provided some assumptions on $f^0$ and on the initial configurations $\mu^0_N$ are satisfied. 

%
\paragraph{\bf Propagation of molecular chaos.}
In many physical settings, the initial positions and
velocities are selected randomly and typically independently (or almost independently).
In the case of total independence, the law of $Z$ is initially given by
$(f^0)^{\otimes N}$, i.e. each couple $Z_i=(X_i,V_i)$ is chosen randomly and 
independently with law
$f^0$. Note that by the empirical law of large number, also known as Glivenko-Cantelli theorem, the empirical measure $\mu_N^\zz(0)$ at time $0$ converges in law to $f^0$ in some weak topology, see for instance Proposition \ref{probaint} for a more precise statement. 

The notion of propagation of chaos was formalized by Kac's in \cite{Kac1956}  and goes
back to Boltzmann and its ``Stosszahl ansatz''. A standard reference is the famous course by Sznitman \cite{Sznitman}. 

Denoting by
$f^N(t,z_1,\ldots, z_N)$ the image by the dynamics
\eqref{eq:ODE} of the initial law $(f^0)^{\otimes N}$, one may define
the $k$-marginals
\[
f^N_k(t,z_1,\ldots,z_k)=\int_{R^{2d(N-k)}}f^N(t,z_1,\ldots,
  z_N)\,dz_{k+1}\dots dz_N.
\]
Propagation of chaos holds when the sequence $f^N(t)$ is $f(t)$-chaotic, \emph{i.e.}  when for any
fixed $k$, $f^N_k(t)$ converges weakly to $[f(t)]^{\otimes k}$ as
$N\rightarrow \infty$. In fact it is sufficient that the convergence holds for $k = 2$. 

It is also equivalent to asking that  the empirical measures
$\mu_N^\zz(t)$ converge in law towards the {deterministic} variable $f(t)$.
This equivalence holds because the marginals can be recovered from
the expectations of moments of the empirical measure
\[
f^N_k={\mathbb E} (\mu_N^\zz(t,z_1)\ldots\mu_N^\zz(t,z_k) ) + O\left(
\frac{k^2}N\right),
\]
a result sometimes called Grunbaum lemma. 

For detailed explanations about
quantification of the equivalence between convergence of the marginals $f^N_k$ and the
convergence in law of the empirical distributions $\mu_N^\zz$, we refer to
\cite{HauMisch}. This quantified equivalence was for instance used in the
recent and important work of Mischler and Mouhot about Kac's program in kinetic
theory \cite{MischMou}.

In the hard sphere problem, propagation of chaos towards the Boltzmann equation (in the Boltzmann-Grad scaling) was shown by Landford~\cite{Landford}, with a non completely correct proof that was completely fulfilled only recently by Gallagher, Saint-Raymond and Texier~\cite{GSRT-LN} (and extended to more general interactions). Unfortunately the deep techniques used in \cite{GSRT-LN} do not seem to be applicable in our case.

We prove in this article deterministic, mean field limit results, see Theorems \ref{thm:deter} and \ref{thm:cutoff}. They then imply quantified versions of the propagation of chaos, in Theorems~\ref{thm:prob} and~\ref{thm:probcutoff}.

\paragraph{Previous results in dimension one.}
Let us shortly mention that in dimension one, the mean field limit and the propagation of chaos are better understood.
In that case, the force $F(x) = \mathrm{sign} (x)$ is ``only''  discontinuous. The first mean field limit result in that case was obtained by Trocheris~\cite{Trocheris}, and it was re-discovered by Cullen, Gangbo and Pisante as a particular case of semi-geostrophic equations~\cite{CGP-ARMA}. We also refer to a simpler proof by the first author~\cite{Hau-X} using a weak-strong stability inequality for the 1D Vlasov-Poisson equation. All these mean-field results imply the propagation of chaos in a straightforward manner.

\paragraph{Previous results with cut-off or for smooth interactions.} 
The mean-field limit and the propagation of chaos are  known to hold for
smooth interaction forces 
($F\in W^{1,\infty}_{loc}$) since the end of
the seventies and the works of Braun and Hepp 
  \cite{BraHep77}, Dobrushin \cite{Dobr79} and Neunzert and Wick \cite{Neun79}.
Those articles introduce the main ideas and the formalism behind
mean field limits; we also refer to the nice book by Spohn \cite{Spoh91}.

Their proofs however rely on Gronwall type estimates and are connected to the
fact that Gronwall estimates are
actually true for \eqref{eq:ODE} uniformly in $N$ if $F\in
W^{1,\infty}$. This makes it impossible to generalize them to any case
where $F$ is singular, including Coulombian interactions and many
other physically interesting models. 

However, by keeping the same general approach, it is possible to deal
with singular interactions with cut-off. For instance for  Coulombian
interactions, one could consider
\[
F_N(x)=C\,\frac{x}{(|x|^2+\ep(N)^2)^{d/2}},
\] 
or other types of regularization at the scale $\ep(N)$. 
The system \eqref{eq:ODE} with such forces does not have much physical meaning but the
corresponding studies are crucial to understand the convergence of
numerical methods.  For particles initially on a regular mesh, we refer
to the works of Ganguly and Victory \cite{Victory},
Wollman \cite{Wollman} and  Batt  \cite{Batt00} (the latter gives a
simpler proof, but valid only for larger cut-off  than in the two first references \Black). Unfortunately they had to
impose that $\lim_{N \rightarrow \infty} \ep(N) N^{1/d} = + \infty$, meaning
that the cut-off for convergence results is usually larger than the one
used in practical numerical simulations. Note that the scale $N^{- 1/d}$ is the
average distance between two neighboring  particles in
position. 

These ``numerically oriented'' results  do not imply the  propagation of chaos, as the particles
are on a mesh initially and hence (highly) correlated.  Moreover, we
emphasize that the two problems with initial particles on a mesh, or with
initial particles not equally distributed seem to be very different.  In the last
case, 
Ganguly, Lee, and
Victory \cite{Victory2}  prove the convergence only for a much larger
cut-off $\ep(N) \approx (\ln N)^{-1}$.

\paragraph{\bf Previous results for $2d$ Euler or other macroscopic equations.} 
A well known system, very similar at first sight with the question here,
is the vortices system for the $2d$ incompressible Euler equation.
One replaces \eqref{eq:ODE} by 
\begin{equation}
\dot X_i=\frac{1}{N} \sum_{j\neq i} \alpha_i\,\alpha_j\,\nabla^\perp\Phi(X_i-X_j),
\label{vortex}\end{equation}
where $\Phi(x) =(2\pi)^{-1} \ln |x|$ is still the Coulombian kernel (in $2$ dimensions here) and
$\alpha_i=\pm1$. One expects this system to converge to the Euler
equation in vorticity formulation
\begin{equation} 
\partial_t \omega +\hbox{div}\,(u\,\omega)=0,\qquad
\hbox{div}\,u=0,\qquad \hbox{curl}\, u=\omega.\label{Euler}
\end{equation}
The same questions of convergence and propagation of chaos can be
asked in this setting. 
Two results without
regularization for the true kernel are already known. The work of
Goodman, Hou and Lowengrub,  \cite{GooHouLow90,Goodman91}, has a
numerical point of view but uses the true singular kernel in a
interesting way. The work of Schochet \cite{Scho96} uses the weak
formulation of Delort of the Euler equation and proves  that empirical
measures with bounded energy converge towards  measures that are weak solutions to
\eqref{Euler}. Unfortunately, the possible lack of uniqueness of the vorticity equation~\eqref{Euler} in the class of measures does not allow to deduce the propagation of
chaos.

{The main difference between \eqref{eq:ODE} and \eqref{vortex} is that}
 System \eqref{eq:ODE} is second
order while \eqref{vortex} is first order. This implies that collisions
or near collisions (in physical space) between particles are very common
for \eqref{eq:ODE} even for repulsive interactions and much less common for \eqref{vortex},
even if vortices of same sign usually tend to merge. 

The references mentioned above use the 
symmetry of the forces in the vortex case; a symmetry which cannot exist
in our kinetic problem, independently of additional structural assumptions like $F=-\nabla\Phi$. The force is still symmetric with respect to the
space variable, but there is now a velocity variable 
which breaks the argument used in the vortices case. For a more complete description of
the vortices system, we refer to the references already quoted or to
\cite{Hau09}, which introduces in that case techniques similar to the
one used here.

%
\paragraph{\bf Our previous result in singular cases without cut-off.} 
To our knowledge, the only mean field limit result available up to now
for System \eqref{eq:ODE} with singular forces is \cite{HauJab07}. We proved the mean field limit (not the
propagation of chaos) provided that:
\begin{itemize}
\item The interaction force $F$ satisfy a $(S^\alpha)$-condition 
with  $\alpha<1$.
\item The particles are initially well distributed, meaning that the
  minimal inter-distance in $\R^{2d}$ is of the same order as the
  average distance between neighboring particles $N^{-1/2d}$.
\end{itemize}

The second assumption is all right for numerical purposes but does not
allow to consider physically realistic initial conditions, as per the
propagation of chaos property. This assumption is indeed
 not generic for empirical measures randomly chosen with law
$(f^0)^{\otimes N}$, {\em i.e.} it is satisfied with probability
going to $0$ in the large $N$ limit. 
%

\paragraph{\bf Organization of the paper.}
In the next section, we state precisely our main theorems.
In the third section, we introduce notations, recall some results on the Vlasov-Poisson equation~\eqref{eq:vlasov} and give a short sketch of the proof. The fourth and longest section is devoted to the proof of the main field limit results,
and we explain in the fifth section why our deterministic results imply the propagation of chaos.
The sixth section contains two important discussions: one about the existence of solution to the system of ODE~\eqref{eq:ODE}, and a second explaining why we cannot use the structure of the force term, when it is of potential form, attractive and repulsive.
Finally, two useful Propositions are proved in the Appendix.
\Black

%
%
\section{Main results}

\subsection{The results without cut-off.}
Our main result in this article is deterministic:  it shows that the mean field limit holds, provided that  interaction forces
still satisfy an $(S^\alpha)$-condition \eqref{eq:Calpha} with $\alpha <1$. The initial distributions of particles have to be uniformly compactly supported, and to satisfy a bound from above on a ``discrete uniform norm'' and again a bound from below on the minimal distance between particles (in position and speed) which is much less demanding than in \cite{HauJab07}. 
\begin{theorem} \label{thm:deter}
Assume that $d \geq 2$ and that the interaction force $F$ satisfies a
$(S^\alpha)$ condition~\eqref{eq:Calpha}, for
some $\alpha <1$ and let $0<\gamma<1$.

Assume that $f^0 \in L^\infty(\R^{2d})$
has compact support and total mass one, and denote by $f$ the unique global, bounded, and compactly supported solution $f$ of the Vlasov equation~\eqref{eq:vlasov}, see Proposition~\ref{prop:WPvlasov}. 

Assume that the initial conditions $Z^0$ are such that for each $N$, there exists a global solution $Z$ to the N particle system~\eqref{eq:ODE}, and that the initial empirical distributions $\mu_N^0$ of the particles satisfy
\begin{itemize}
\item[i)] For a constant $C_\infty$ independent of $N$,
$$ 
\sup_{z \in \R^{2d}} N^\gamma \mu_N^0 \Bigl( B_{2d}\bigl(z, N^{-\frac \gamma{2d}} \bigr)\Bigr) \leq C_\infty, 
\quad \text{and} \quad
\| f_0 \|_\infty \le C_\infty;
$$ 
\item[ii)] For some $R_0 >0$, $\forall N \in \N, \quad \Supp \, \mu_N^0 \subset B_{2d}(0,R_0)$;

\item[iii)]  for some $r
\in(0,r^*)$ where $r^*:=\frac{d-1}{1+\alpha}$,
$$
\inf_{i\neq j} |(X_i^0,\,V_i^0)-(X_j^0,\,V_j^0)|
  \geq N^{-\gamma(1+r)/2d}.
$$
\end{itemize}
Then for any $T>0$, there exist two constants $C_0(R_0,C_\infty,F,T)$ and
$C_1(R_0, C_\infty,F,\gamma,r,T)$ such that for $N \geq e^{C_1 T}$ the following
estimate holds
\begin{equation} \label{eq:thm1}
 \forall t \in [0,\ T], \quad W_1(\mu_N(t),f(t))  \leq
e^{C_0t}  \Bigl( W_1(\mu^0_N,f^0) +  2\, N^{-\frac\gamma{2d}} \Bigr),
\end{equation}
where $W_1$ denotes the $1$ Monge-Kantorovitch-Wasserstein distance.
\end{theorem}

\begin{remark} 
The condition $(i)-(iii)$ are fulfilled when the initial positions and velocities of the particles are chosen on a mesh.  They are also fulfilled when one considers a finite number of particles inside cells of a mesh, as it is usually done in PIC method. 
\end{remark}

To deduce from the previous theorem the propagation of chaos, it remains to show that we can apply its deterministic stability result to most of the random initial conditions. Precisely, we can show that when the initial positions and velocities are i.i.d. with law $f^0$, then the conditions $(i)-(iii)$ of Theorem~\ref{thm:deter}  are satisfied with a probability going to one in the limit,
This leads to
a quantitative version of propagation of chaos. 
\begin{theorem} \label{thm:prob}
Assume that $d \geq 3$ and that $F$ satisfies a $(S^\alpha)$-condition \eqref{eq:Calpha} with $\alpha < 1$. There exist a positive real number $\gamma^* \in(0,1)$ depending only on $(d,\alpha)$ and a function $s^\ast : \gamma \in (\gamma^\ast,1) \rightarrow s^\ast_\gamma \in (0,\infty)$  s.t.:

- For any non negative initial data $f^0 \in L^\infty(\R^{2d})$
with compact support and total mass one, denoting by $f$ the unique global, bounded, and compactly supported solution $f$ of the Vlasov equation~\eqref{eq:vlasov}, see Proposition~\ref{prop:WPvlasov}; 

- For each $N \in \N^*$,  denoting by $\mu_N$ the empirical measure corresponding to the solution to \eqref{eq:ODE} with initial positions $Z^0= (X_i^0,V_i^0)_{i \leq N}$ chosen randomly according to the probability
$(f^0)^{\otimes N}$;  

Then, for all $T >0$, any
\[
\gamma^*< \gamma <1 \quad \text{and} \quad
0<s<s_\gamma^*, 
\] 
there exists three positive constants $C_0(T,f,\Phi)$, $C_1(\gamma,s,T,f,\Phi)$ and $C_2(f^0,\gamma)$ such that for $N \ge e^{C_1 T}$ 
\begin{equation} \label{eq:thmprob}
   \PP
\left(  \exists \, t \in [0,T], \; W_1(\mu_N^\zz(t), f(t)) \geq 3 \, e^{C_0t} \, N^{- \frac\gamma{2d}}  \right)
\leq  \frac {C_2} {N^s}.
\end{equation}
The constants $C_1$ and $C_2$ 
blow up when $\gamma$ or $s$ approach their maximum value.
\Black
\end{theorem}
\begin{remark} We have explicit formulas for $\gamma^*$ and $s_\gamma^*$ namely
\begin{equation}\label{eq:explicit}
\gamma^\ast := \frac{2+2\alpha}{d+\alpha} 
\quad \text{and} \quad
s_\gamma^\ast := \frac{\gamma d -(2-\gamma)\alpha-2 }{2(1+\alpha)}.
\end{equation}
Those conditions are not completely obvious, but it
can be checked that if $\alpha <1$ and  $d \geq 3$,
$\gamma^\ast<1$ so that admissible $\gamma$
exist. And for  an admissible $\gamma$, $s_\gamma^\ast$ is also positive, so that
admissible $s$ also exists. The best choices for $\gamma$ and $s$ would be 
$\gamma=1$ and $s=\frac{d-\alpha-2}{2(1+\alpha)}$ as those give the fastest convergence. Unfortunately the constant $C_1$ and $C_2$ would then be $+\infty$ hence the more complicated formulation. 
\end{remark}
\begin{remark}
Roughly speaking, under
the assumptions of Theorem \ref{thm:prob}, except for a small set of initial conditions $\mathcal S_N^c$, 
the deviation between the empirical measure and the limit is at most of the same order as the average inter-particle distance
$N^{-1/2d}$.
\end{remark}
\begin{remark} 
The deterministic Theorem~\ref{thm:deter} is valid in dimension $2$. 
Unfortunately, 
its assumptions are not generic in dimension $2$ for initial conditions
chosen randomly and independently. This is why we cannot prove the propagation of chaos
for $d=2$ in Theorem~\ref{thm:prob} even for small $\alpha$. In fact, note for instance that if $d=2$ then $\gamma^*$ defined in~\eqref{eq:explicit} is larger than $1$ so that it is never possible to find $\gamma$ in $(\gamma^*,\;1)$.
\end{remark}

\begin{remark} \label{rmk:exist}
The arguments in the proof of Theorem~\ref{thm:prob} prove that, at fixed $N$, there exists a global solution to~\eqref{eq:ODEint} for a large set of initial conditions. In fact, in a very sketchy way, this theorem also propagates a control on the minimal inter-particles
distance in position-velocity space. Used as is, it only says that asymptotically, the control is good with large probability. However for fixed $N$, if we let some constants increase as much as needed, it is possible to modify the argument and obtain a control for almost all initial configurations. Since the proof also implies that the only bad collisions are the collisions with vanishing relative velocities, we can obtain  existence (and also uniqueness) for almost all initial data of the ODE~\eqref{eq:ODE}.
\end{remark}

\paragraph{The improvements with respect to \cite{HauJab07}.} The major
improvement is the much weaker assumption in Theorem \ref{thm:deter} on the initial
distribution of positions and velocities, which enables us to prove
the propagation of chaos. 

The method of the proof is also quite different. It now relies on explicit bounds between the empirical measure and an appropriate solution to the limit equation \eqref{eq:vlasov}. This lets us easily use the properties of \eqref{eq:vlasov}, and dramatically simplifies the proof in the long time case which 
was very intricate in \cite{HauJab07} and does not require any special treatment here.  

Finally, our analysis is now quantitative: 
For large enough $N$, Theorem \ref{thm:deter} gives a precise rate of
convergence in Monge-Kantorovitch-Wasserstein distance, with important applications from the point of view of the numerical analysis (giving rates of convergence for particles' methods for instance).
For more details about the novelties and improvements with respect to \cite{HauJab07}, we refer to the Sketch of the proof in Subsection~\ref{subsec:sketch}.
 
Unfortunately, the condition on the interaction force $F$ is still the same and
does not allow to treat Coulombian interactions. There are some
physical reasons for this condition, which are discussed at the end of the article in subsection~\ref{subsec:sign}. We
refer to \cite{BaHaJa} for some ideas in how to go beyond this threshold in the repulsive case. 

%
%
\subsection{The results with cut-off.} 
The result  presented here is in one
 sense slightly weaker than the previously
known result \cite{Victory2}, since we just miss the critical case
$\alpha = d-1$. But in that work the cut-off used is very large: $\ep(N) \approx (\ln N)^{-1}$. Instead we are able to use cut-off that are some power of $N$ and much more realistic from a physical point of view. 
 For instance, astrophysicists doing gravitational
 simulations ($\alpha = d-1$) with ``tree codes'' usually use small cut-off
 parameters, lower than $N^{-1/d}$ by some order. 
See \cite{Dehn00} for a physical oriented discussion about the optimal length of
this parameter.     
\begin{theorem} \label{thm:cutoff}
Assume that $d \geq 2$ and that the interaction force $F_N$ satisfies a $(S^\alpha_m)$ condition~\eqref{eq:Ckappa}, for some $1 \leq \alpha < d-1$, with a cut-off order
satisfying
\[
m <  m^* := \frac 1{2d}
\min \left(\frac{d-2}{\alpha -1} \,, \frac{2d-1}{\alpha} \right),
\]
and choose any $\gamma \in \bigl( \frac m{m^*}, 1 \bigr)$.

Assume that $f^0 \in L^\infty(\R^{2d})$
with compact support and total mass one, and denote by $f$ the unique, bounded, and compactly supported solution $f$ of the Vlasov equation~\eqref{eq:vlasov} on the maximal time interval $[0,\ T^*)$, see Proposition~\ref{prop:WPvlasov}. 

Assume also that for any $N$, the initial empirical distribution of the particles
$\mu_N^0$
satisfies:
\begin{itemize}
\item[i)] For a constant $C_\infty$ independent of $N$,
$$ 
\sup_{z \in \R^{2d}} N^\gamma \mu_N^0 \Bigl( B_{2d}\bigl(z, N^{-\frac \gamma{2d}} \bigr)\Bigr) \leq C_\infty, 
\quad \text{and} \quad
\| f_0 \|_\infty \le C_\infty;
$$ 
\item[ii)] For some $R_0 >0$, $\forall N \in \N, \quad \Supp \,\mu_N^0 \subset
B_{2d}(0,R_0)$.
\end{itemize}
Then for any time $T<T^*$,
there exist $C_0(R_0,C_\infty,F,T)$ and
$C_1(R_0, C_\infty,F,\gamma,r,T)$ such that for $N \geq e^{C_1 T}$ the following
estimate holds
\begin{equation} \label{eq:thm3}
 \forall t \in [0,\ T], \quad W_1(\mu_N(t),f(t))  \leq
e^{C_0t}  \Bigl( W_1(\mu_N^0,f_N0) +  3\, N^{-\frac\gamma{2d}} \Bigr).
\end{equation}
\end{theorem}
\begin{remark}
One would like to take $m$ as large as possible if
we want to be close to the dynamics without cut-off. 
\end{remark}
\begin{remark}
Theorem \ref{thm:cutoff} result is also interesting for numerical simulations
because one obvious way to fulfill the assumption on the infinite norm of
$f_N^0$ is to put particles initially on a mesh (with a grid length of
$N^{-1/2d}$ in $\R^{2d}$). In that case, the result is even valid with $\gamma
=1$.
\end{remark}
As in the case without cut-off, the fact that the mean-field limit holds under ``generic'' conditions implies the propagation of molecular chaos.
\begin{theorem} \label{thm:probcutoff}
Assume that $d \geq 3$ and that $F_N$ satisfies a
$(S^\alpha_m)$-condition for some  $1 \leq \alpha< d-1$ with a cut-off
order $m$ such that 
\[ 
m < m^* := \frac 1{2d}
\min \left(\frac{d-2}{\alpha -1} \,, \frac{2d-1}{\alpha} \right),
\] 
and choose any $\gamma \in \bigl( \frac m{m^*}, 1 \bigr)$.

Choose any initial condition $f^0 \in L^\infty$ with
compact support and total mass one for the Vlasov equation \eqref{eq:vlasov},  and denote by $f$ the unique strong solution of the Vlasov equation\eqref{eq:vlasov} with initial condition $f^0$ on the maximal time interval $[0,\ T^*)$, given by Proposition~\ref{prop:WPvlasov}. 

For each $N \in \N^*$, 
consider the particles system~\eqref{eq:ODE} for $F_N$ with initial positions $(X_i,V_i)_{n \leq N}$ chosen randomly according to the probability
$(f^0)^{\otimes N}$. 

Then for any time $T<T^*$, there exist positive constants
$C_0(T,f,\Phi)$,  $C_1(\gamma,m,T,f,\Phi)$ $C_2(f)$ and $C_3(f)$  such that for $N \ge e^{C_1 T}$
\[
 \PP
\left(  \exists t \in [0,T], \; W_1(\mu_N(t), f(t)) \geq 4\, e^{C_0t}N^{-\frac \gamma{2d}} 
\right)
\leq  C_2 N^\gamma \, e^{ - C_3 N^\lambda}  ,
\]
where  $\lambda = 1 - \max \left( \gamma, \frac1d \right) $.
\end{theorem}

\begin{remark}
Our result is valid only locally in time (but on the largest interval of time possible)
in the case where blow-up may occur in the Vlasov equation, as for instance  in dimension larger than four with attractive interaction. But it is valid for any time in dimension three or less, since in that case 
the strong solutions of the Vlasov equations we are dealing with are global, see Proposition~\ref{prop:WPvlasov} in section~\ref{sec:uniq}. 
\end{remark}

\subsection{Open problems and possible extensions}
In dimension $d=3$, the minimal cut-off is given by 
$m^\ast = \frac \gamma 6 \min ( (\alpha -1 )^{-1}, 5 \alpha^{-1}
)$. As $\gamma$ can be chosen very close to one, for $\alpha$ larger
but close to one, the previous bound tells us that we can choose
cut-off of order almost $N^{-5/6}$, i.e. much smaller than the likely minimal
inter-particles distance in position space ( of order $N^{-2/3}$, see
the third section).  With such a small cut-off, one could hope that it is
almost never used when we calculate the interaction forces between
particles. Only a negligible number of particles will become that close to
one another before the time $T$. This suggests that there should be
some way
to extend the result of convergence without cut-off at least to some
$\alpha>1$.

Unfortunately, we do not know how to make rigorous
 the previous argument on the close
encounters. First it is highly difficult
to translate for particles system that are highly correlated. To state
it properly we need $L^\infty$ bounds on the $2$-particle marginal. But
obtaining such a bound for singular interactions seems
difficult. Moreover, it remains to control the influence of particles
that have had a close encounters (their trajectories after a encounter are
not well controlled)  on the other particles. 

\paragraph{Many particles systems with diffusion.}

It would be very natural to try to adapt our techniques to the stochastic case of Langevin equations
\begin{equation} \label{eq:Langevin}
\forall i \le N, \quad 
\begin{cases} \displaystyle 
   X_i(t) = X_i^0 + \int_0^t V_i(s) \,ds  \\
\displaystyle   V_i(t) = V_i^0 + \frac1N \sum_{j \neq i} \int_0^t F(X_i(s) -
X_j(s))\,ds - \lambda \int_0^t V_i(s) \,ds + \nu B_i(t),
\end{cases} \end{equation} 
where the $B_i$ are independent Brownian motions, and $\nu,\lambda >0$. Solutions of that system should formally converge to solutions of the Jeans-Vlasov-Fokker-Planck equation
\begin{equation}   \label{eq:vlasovFP}
\begin{cases}
& \ds \partial_t f + v \cdot \nabla_x f +  E(x) \cdot \nabla_v f = \frac{\nu^2}2 \Delta_v f  + \lambda \diver( vf)\,, \\
& \ds E(x) = \int_{\R^d} \rho(t,y)\,F(x-y)\,dy.
\end{cases} \end{equation}
It was shown by McKean in~\cite{McKean} that the propagation of chaos holds when $F \in W^{1,\infty}$. But to the best of our knowledge, there is not any similar result when the interaction force is singular, even weakly.
Our techniques, which rely on strong controls on the trajectories and on the minimal inter-particle distance are very sensitive to noise, and (at least) cannot be directly adapted to the stochastic case.

Remark that the situation is in some way ``opposite'' in the vortex case.  The propagation of chaos for the stochastic vortex system (the system~\eqref{vortex} with independent noises) was first proved by Osada in the eighties~\cite{Osada}, and recently generalized by Fournier, Mischler and the first author~\cite{FHM-JEMS}.

%
\section{Notation, useful results and sketch of the proof.}   \label{sec:main}
 
\subsection{Notation} 
  
  In the sequel, we always use the Euclidean distance on $\R^d$ for positions or velocities, or on $\R^{2d}$ for  couples ``position-velocity''. In all case, it will be denoted by $|x|$, $|v|$, $|z|$.
The notation $B_n(a,R)$ will always stand for the ball of center $a$ and radius $R$ in dimension $n=d$ or $2d$. The Lebesgue measure of a measurable set $A$ will also be denoted by $|A|$.

\medskip
{\bf \textbullet~ Empirical distribution $\mu_N$ and minimal inter-particle
  distance $d_N$} 

Given a configuration $Z = (X_i,V_i)_{i \leq N}$ of the particles in the
phase space $\R^{2dN}$,  the associated empirical distribution is the
measure 
\[
\mu_N^\zz = \frac1N \sum 
  \delta_{X_i,V_i}.
\]
An important remark is that if
  $(X_i(t,),V_i(t))_{i \leq N}$ is a solution of the system of ODE
  \eqref{eq:ODE}, then the measure $\mu_N^\zz(t)$ is a solution of the
  Vlasov equation \eqref{eq:vlasov} in a weak sense, provided that the interaction
  force satisfies $F(0)=0$. This condition is necessary to avoid
  self-interaction of Dirac masses.  It means that the interaction
  force is defined everywhere, but discontinuous and has a singularity
  at $0$. 

For every empirical measure, we define the minimal distance
$d_N^\zz$  between particles in $\R^{2d}$  
\begin{equation} \label{eq:dmin}
d_N^\zz= d_N(\mu^\zz_N) :=  \min_{i \neq j} |Z_i - Z_j| = \min_{i \neq j} \bigr(|X_i - X_j|^2 + |V_i-V_j|^2 \bigl)^{\frac12}.  
\end{equation}
This is a non physical quantity, but it is crucial to control the
possible concentrations of particles and we will need 
to bound that quantity from below.  

In the following we often omit the $Z$ superscript,  in order to
keep "simple" notations.

\medskip{\bf \textbullet~ Infinite MKW distance }   

We use many times the
  Monge-Kantorovitch-Wasserstein distances of order one and
  infinite. The order one distance, denoted by $W_1$, is classical and
  we refer to the very clear book of Villani for definition and properties
  \cite{Vill03}. The second one denoted $W_\infty$ is not widely
  used,  so we recall its definition. We start with the definition of transference plane
\begin{defn}
Given two probability measures 
$\mu$ and $\nu$ on $\R^n$ for any $n \ge 1$, a transference plane $\pi$ from $\mu$ to $\nu$ is a probability measure on $X\times X$ s.t.
\[
\int_X \pi(dx,dy)=\mu(dx),\qquad \int_X \pi(dx,dy)=\nu(dy),
\] 
that is the first marginal of $\pi$ is $\mu$ and the second marginal is $\nu$. 
\end{defn}  
With this we may define the $W_\infty$ distance
\begin{defn}  \label{def:Winf}
For two probability measures 
$\mu$ and $\nu$ on $\R^n$, with $\Pi(\mu,\nu)$
the set of transference planes from $\mu$ to $\nu$: 
\[
W_\infty(\mu,\nu) = \inf \; \bigl\{ \pi-\hbox{esssup}\, |x-y|\;  \big\vert  \; \pi
\in \Pi  \bigr\}. 
\]
\end{defn}
There is also another notion, called the
transport map. A transport map is a measurable map $T : \Supp \, \mu
\rightarrow \R^n$ such
that $(Id,T)_\# \mathcal \mu \in \Pi$. This means in particular that $T_\# \mu
= \nu$, where the pushforward of a measure $m$ by a transform $L$ is defined by
\[
L_\# m(O)=m(L^{-1}(O)),\qquad\mbox{for any measurable set}\ O.
\]
 In one of the few works on the subject \cite{ChaDePJuu08} Champion, and De
Pascale and Juutineen prove that if
$\mu$ is absolutely continuous with respect to the Lebesgue measure $\mathcal
L$, then at least one optimal transference plane for the infinite MKW
distance is given by a optimal transport map, i.e. there exists $T$ s.t.
$(Id,T)_\# \mathcal \mu \in \Pi$ and
\[
W_\infty(\mu,\nu)= \mu-\hbox{esssup}_x\, |Tx-x|. 
\]
Although that is not mandatory (we could actually work with optimal
transference planes), we will use this result and work in the sequel with
transport maps. That will greatly simplify the notations in the proof. 

\medskip
Optimal transport is useful to
compare the discrete sum appearing in the force induced by the  N particles to the integrals of the
mean-field force appearing in the Vlasov equation. For instance, if $f$ is a continuous distribution and
$\mu_N$ an empirical distribution we may rewrite the interaction force of
$\mu_N$ using a transport map $T= (T_x,T_v)$ of $f$ onto $\mu_N$ 
\[
\frac1N \sum_{i \neq j} F(X_i^0-X_j^0)  = \int F( X_i^0 - T_x(y,w)) f(y,w)
\,dydw. 
\]
Note that in the equality above, the function $F$ is
singular at $x=0$, and that we impose $F(0)=0$. The interest of the infinite MKW distance is that the singularity is still localized ``in a ball'' after the transport : The term under
the integral in the right-hand-side  has no singularity out of a ball of radius
$W_\infty(f,\nu_N)$ in $x$.  Other MKV distances of order $p< + \infty$ destroy
that simple localization after the transport, which is why it
 seems more difficult to use them.

\medskip
{\bf \textbullet~ The scale $\ep$.}
We also introduce  a scale 
\begin{equation} \label{eps}
 \ep(N) = N^{-\gamma/2d} \,, 
\end{equation}
for some $\gamma \in (0,1)$ to be fixed later but close enough
from $1$. Remark that this scale is larger than the average distance between a
particle and its closest neighbor, which is of order $N^{-1/2d}$. We will often define quantities
directly in term of $\ep$ rather than $N$. For instance, the  cut-off order $m$
used in the $(S^\alpha_m)$-condition may be rewritten in term of $\ep$, with
$\bar m := \frac{2d}\gamma m  \in \bigl(1, \min(\frac{d-2}{\alpha-1}, \frac{2d-1}\alpha) \bigr)$.
\begin{equation} \label{eq:Ckappa'} 
(S^\alpha_m) \qquad 
\begin{array}{ll}
i)  &   F \; \text{satisfy a } (S^\alpha)-\text{condition}, \\
ii) &   \forall \, |x| \geq \ep^{\bar m}, \,  F_N(x) = F(x), \\
iii)&    \forall \, |x| \leq \ep^{\bar m}, \,  |F_N(x) | \leq \ep^{ -\bar{m}
\alpha}.
\end{array} \end{equation}

\medskip
{\bf \textbullet~ The solution $f_N$ of Vlasov equation with blob initial
condition.}

Now we defined a smoothing of $\mu_N$ at the scale $\ep(N)$. For this, we choose
a kernel $\phi: \R^{2d} \rightarrow \R$ radial with compact support in $B_{2d}(0,1)$ and total mass one, and denote $\phi_\ep(\cdot) = \ep^{-2d}
\phi(\cdot/\ep)$. The precise choice of $\phi$ is not very relevant, and the
simplest one is maybe $\phi = \frac1{|B_{2d}(0,1)|} {\mathbf 1}_{B_{2d}(0,1)}$.
We use this to smooth $\mu_N$ and define 
\begin{equation} \label{eq:deffN}
f_N^0 = \mu_N^0 \ast \phi_{\ep(N)},
\end{equation}
and denote by $f_N(t,x,v)$ the solution to the Vlasov Eq. \eqref{eq:vlasov}
for the initial condition $f_N^0$.

With $f_N$, the assumption of point $i)$ in Theorems~\ref{thm:deter} and~\ref{thm:cutoff} may be rewritten
$$
\| f^0_N \|_\infty \leq C_\infty,
$$
independently of $N$. 
And this also holds for any time since $L^\infty$ bound are propagated by the Vlasov equation. That $L^\infty$ bound allows to use standard stability estimates to control its $W_1$
distance to another solution of the Vlasov equation, see Loeper result
\cite{Loep06} recalled in Proposition~\ref{Loeper}. 

A key point in the rest of the article is that $f_N^0$ and $\mu_N^0$ are very close in $W_\infty$ distance as per
\begin{prop} \label{prop:Winfbound} For any 
$\phi: \R^{2d} \rightarrow \R$ radial with compact support in $B_{2d}(0,1)$ and total mass one
we have for any $\mu_N^0=\frac{1}{N}\sum_{i=1}^N \delta_{(X_i^0,V_i^0)}$
\[
 W_\infty (f_N^0,\mu_N^0) = c_\phi \ep(N)
\]
where $c_\phi$ is the smallest $c$ for which $\Supp \, \phi \subset \overline{B_{2d}(0,c)}$.
\end{prop}
\begin{proof}
Unfortunately even in such a simple case, it is not possible to give a simple explicit formula for the optimal transport map. But there is a rather simple optimal transference plane. Define
\[
\pi(x,v,y,w)=\frac{1}{N}\sum_{i} \phi_\ep(x-y,v-w)\,\delta_{(X_i^0,\,V_i^0)}(y,w).
\] 
Note that
\[
\int_{\R^{2d}} \pi(x,v,dy,dw)= [\mu_N^0 \ast \phi_\ep ](x,v) = f_N^0(x,v),
\]
and since $\phi_\ep$ has mass $1$
\[
\int_{\R^{2d}} \pi(dx,dv,y,w)=\frac{1}{N}\sum_{i} \delta_{(X_i^0,\,V_i^0)}(y,w) = \mu_N^0(y,w).
\]
Therefore $\pi$ is a transference plane between $f^0_N$ and $\mu_N^0$. Now take any $(x,v,y,w)$ in the support of $\pi$. By definition there exists $i$ s.t. $y=X_i^0$, $w=V_i^0$ and $(x,v)$ is in the support of $\phi_\ep(.-X_i^0,v-V_i^0)$. Hence by the assumption on the support of $\phi$
\[
|x-y|^2+|v-w|^2\leq  c_\phi [\ep(N)]^2,
\] 
which gives the upper bound. 

We turn to the lower bound. Remark  that the assumptions imply that $\phi>0$ on $B_{2d}(0,c_\phi)$.
Choose $X_i^0,\;V_i^0$ any extremal point of the cloud $(X_j^0,V_j^0)_{j \le N}$. Denote $u_i\in S^{2d-1}$ a vector separating the cloud at $X_i^0,\;V_i^0$, {\em i.e.}
\[
u_i\cdot (X_j^0-X_i^0,\;V_j^0-V_i^0)<0,\qquad \forall j\neq i.
\]
Now define $(x,v)=(X_i^0,\;V_i^0)+ \lambda \,\ep(N)\,u_i$. Since $\phi_\ep$ is radial and $\phi_\ep>0$ on $B(0,c_\phi\,\ep)$ then $f_N^0(x,v)>0$ when $\lambda < c_\phi$. Denote by $T$ the optimal transference map. $T(x,v)$ has to be one of the $(X_j^0,\;V_j^0)$. Hence by the definition of $u_i$, $|(x,v)-T(x,v)|\geq \lambda \,\ep(N)$. Since it is true for any $\lambda < c_\phi$, and for any $\tilde u$ in a neighborhood of $u$, it implies that $f_N^0-\mathrm{esssup} \, |T - Id| \ge c_\phi \ep(N)$. That last argument may be adapted if we use an optimal transference plane, rather than a map. This means in particular than the plane $\pi$ defined above is optimal.  But it is not the only one, except if the blobs never intersect. 
\end{proof}

\medskip
Before turning to the proof of our results on the mean field limit, we give some results about the existence and uniqueness of strong solutions to the Vlasov equation~\eqref{eq:vlasov}.
%
%
\subsection{Uniqueness, Stability of solutions to the Vlasov equation~\ref{eq:vlasov}.} \label{sec:uniq}

The already known results about the well-posedness (in the strong sense) of the Vlasov equation that we are considering  are gathered in the following proposition. 
\begin{prop}\label{prop:WPvlasov}
For any dimension $d$, and any $\alpha \le d-1$, and any compactly supported and bounded initial condition $f^0$ there exists a unique local (in time) strong solution to the Vlasov equation~\eqref{eq:vlasov} that remains bounded and compactly supported. In general, the maximal time of existence $T^\ast$ of this solution may be finite, but in the two particular cases  below we have $T^\ast = + \infty$ :
\begin{itemize}
\item $\alpha <1$ (and any $d$),
\item $d \le 3$, and $\alpha \leq  d-1$.
\end{itemize}
In the other cases, the maximal time of existence of the strong solution may be bounded by below by some constant depending only on the $L^\infty$ norm and the size of the support of the initial condition.
The size of the support at any time $t$ may also be bounded by a constant depending on the same quantities.
\end{prop}

The local existence part in Proposition~\ref{prop:WPvlasov} is a consequence of the following Lemma which is proved in the Appendix and the following Proposition~\ref{Loeper} 
\begin{lemma} \label{lem:KeyEstim}
Let $f\in L^\infty([0,\ T],\ \R^{2d})$ with compact
support  be a solution to \eqref{eq:vlasov} in the sense of
distribution with an $F$ satisfying an $(S^\alpha)$ condition\eqref{eq:Calpha} with $\alpha \le d-1$. 
Then if we denote by $R(t)$ and $K(t)$ the size of the
supports of $f$ in space and velocity, they satisfy for a numerical
constant $C$
\[\begin{split}
&R(t)\leq R(0)+\int_0^t K(s)\,ds,\\
&K(t)\leq
K(0)+C\,\|f(0)\|_{L^\infty}^{\alpha/d}\,\|f(0)\|_{L^1}^{1-\alpha/d}
\int_0^t K(s)^\alpha\,ds.
\end{split}
\]
\label{lemstrsol}
\end{lemma}

The local uniqueness part in Proposition~\ref{prop:WPvlasov}
is a consequence of the following stability estimate proved in  \cite{Loep06} for $\alpha = d-1$. Its proof
may be adapted to less singular case. For instance, the adaptation is done in
\cite{Hau09} in the Vortex case.
\begin{prop}[From Loeper] \label{Loeper}
 If $f_1$ and $f_2$ are two solutions of Vlasov Poisson equations with
different interaction forces $F_1$ and $F_2$ both satisfying a
$(S^\alpha)$-condition, with $\alpha <d-1$, then
$$
\frac d {dt} W_1(f_1(t),f_2(t) ) \leq C
\max(\|\rho_1\|_\infty,\|\rho_2\|_\infty) \,\bigl[ W_1(f_1(t),f_2(t))
+  \| F_1 - F_2 \|_1 \bigr] $$
\end{prop}

In the case $\alpha =d-1$, Loeper only obtain in \cite{Loep06} a "log-Lip" bound and not a linear one, but it still implies the stability. 

\medskip
Finally, the global character of the solution in Proposition~\ref{prop:WPvlasov} is  : 
\begin{itemize}
\item a consequence of the lemma~\ref{lem:KeyEstim} if $\alpha<1$, since in that case, the estimates obtained in that lemma show that $R(t)$ and $K(t)$ cannot blow-up in finite time,
\item a much more delicate issue in the case $d \le 3$, and $\alpha = d-1$, finally solved in \cite{LioPer91}, \cite{Scha91} and \cite{Pfaf}. Their proofs may also be extended to the less singular case \mbox{$\alpha < d-1$}.
\end{itemize}

%
\subsection{A short sketch of the proofs.\label{subsec:sketch}}
Here we give a short sketch of the proof. We give only ``almost correct'' ideas, and refer to the proof for fully correct statements. We put the emphasis on the novelty with respect to our previous work~\cite{HauJab07}.
We concentrate mostly on the proof of Theorem~\ref{thm:deter}: the proof of Theorem~\ref{thm:cutoff} is very similar and simpler, and we say only a few words about the propagation of chaos at the end.

\medskip
We use some notations: 
\begin{itemize}
\item $d_N(t) = \inf_{i\neq j} |Z_i(t) -Z_j(t)|$ is the minimal distance between particles. By assumption, it is roughly of order $\ep^{1+r}$ at time $0$.
\item $\Winf (t)= \Winf \bigl(\mu_N(t),f_N(t) \bigr)$, the infinite Monge-Kantorovitch-Wasserstein distance,  see Section~\ref{sec:main}. $\Winf(0)$ is by construction of order $\ep$.
\item In order to deals with quantities of order one, we also introduce 
{($r$ is defined in Theorem~\ref{thm:deter})}
$$
\tilde d_N(t) := \ep^{-(1+r)} \, d_N(t), \qquad 
\tWinf(t) := \ep^{-1} \, \Winf(t).
$$
\end{itemize}

As mentioned above, the Vlasov equation \eqref{eq:vlasov} is
satisfied by the empirical distribution $\mu_N$ of the interacting particle
system provided that $F(0)$ is set to $0$. Hence the problem of
convergence can be reformulated into
 a problem of stability of the empirical measures $\mu_N(t)$ -
seen initially as measure valued 
perturbations of the smooth profile $f^0$ - around
the solution $f(t)$ of the Vlasov equation. 
The proof of the two mean-field limit results use two ingredients to obtain this stability:
\begin{itemize}

\item A standard stability estimate (See Proposition~\ref{Loeper}) for solution of the Vlasov-Poisson equation~\eqref{eq:vlasov}, (with the $1$  Monge-Kantorovitch-Wasserstein distance $W_1$):
$$
W_1\bigl(f_N(t),f(t) \bigr) \le  e^{C t}  W_1\bigl(f_N^0,f^0 \bigr), \quad C:= \sup_{s \le t} \bigl( \|\rho_f(s)\|_\infty+ 
\|\rho_{f_N}(s)\|_\infty \bigr).
$$
\item A control on $\Winf(t)$  (remark that we always have $W_1 \le \Winf$). 
\end{itemize}

Once this will be achieved, we will get a quantitative control on the rate of convergence. This is an important improvement with respect to~\cite{HauJab07}, where we used a compactness argument to prove the convergence and did not get any convergence rate.\Black
We emphasize that
the use of the infinite MKW distance is important. We were not able to perform
our calculations with other MKW distances of order $p < + \infty$  
as the infinite distance is the only MKW distance with which we can handle a localized singularity in the force and Dirac masses in the empirical distribution.    

\medskip

The control on $\Winf(t)$ requires to estimate the difference between the force terms acting 
in the two systems (the particle system and the continuous distribution $f_N$).
Precisely, we need to compare short average on time interval of length $\ep$ of the forces:
$$
\tilde E_N(t,i)  = \frac1N \sum_j \int_{t-\ep}^t F\bigl(X_i(s)-X_j(s)\bigr) \, ds, \quad 
\tilde E_\infty(t,z) = \int_{\R^d} \int_{t-\ep}^t  F(x_s-y)  f(s,y,w) \,dydw\,ds,
$$
when $Z_i=(X_i,V_i)$ and $z=(x,v)$ are close ($x_s$ denotes the position 
at time $s$ of the point starting at $(t,z)$ when following the characteristics defined by $f_N$).
For this comparison, it is necessary to distinguish the contributions of three domains:

\medskip \hspace{10pt} \textbullet  \hspace{5pt} 
Contribution of particles $j$ (and point $y$) far enough from $X_i$ and $x$ in the physical space. 
This is the simplest case as one does not see the discrete
  nature of the problem at that level. The estimates need to be adapted to the $W_\infty$ distance used here 
but are otherwise very similar in spirit to the  continuous problem or other previous works for mean field limits.

\medskip \hspace{10pt} \textbullet  \hspace{5pt} 
 Contribution of particles $j$ (and point $y$) $\ep$-close in the physical
  space $\R^d$ to $X_i$ and $x$, but with sufficiently different velocities. It corresponds to a domain of volume of order $\ep^d$, but where the force is singular.  Here we start to see the discrete level of the problem and
  in fact we cannot compare anymore the discrete and continuous forces:  Instead we
  just show that both are small. The continuous force term is handled easily, but the
  discrete force term requires more work: the short average in time is really required to get rid of possible singularities.

Precisely, consider a second particle $j\neq i$, and neglect the variation of velocities on $[t-\ep,t]$. Because of \eqref{eq:Calpha}, with $\alpha <1$, we have 
\[
\int_{t-\ep}^t |F(X_i(s)-X_j(s))|\,ds\sim \int_{t-\ep}^t \frac{ds}{|\delta+(s-s_0)(V_i-V_j)|^\alpha} 
\lesssim \frac{\ep^{1-\alpha}}{|V_i-V_j|^{-\alpha}}
\]
where $\delta$ is the minimum distance between the two particles on the time interval $[t-\ep,t]$, which is reached at time $s_0$.
The full contribution is obtained after a careful summation on all the particles $j$ of the domain. 

There is here a major improvement with respect to~\cite{HauJab07}. In this previous work bounding the number of particles in that domain was straightforward, since  we assumed that  $\ep \le d_N$ (that bound was propagated in time) so that particles were mostly equi-distributed at scale $\ep$. Instead here, we use the $L^\infty$ bound on $f_N$ and the $\Winf$ distance to obtain a control of the contribution of all these particles, which is more delicate. 

\medskip \hspace{10pt} \textbullet  \hspace{5pt} 
 Contribution of particles $\ep$-close in $\R^{2d}$, \emph{i.e.} in position and velocity. 
 This a very small domain, of volume of order $\ep^{2d}$, but it contains particles that are close in physical space and are likely to remain close for a rather long time (small relative velocity).
 
 Again, there is a major improvement with respect to \cite{HauJab07}, as this case was relatively simple there: under our restrictive assumption on $d_N$ that last domain contained only a bounded number of particle. Here the lower bound on $d_N$ is much smaller, of  order $\ep^{1+r}$. It is even surprising that it is possible to control $d_N$ at a scale which is much lower than the natural discrete scale of the problem.
The key to this new control is due to the fact that the ODE system is second order so that the trajectories (in position space) can be approximated by straight lines up to second order in time, thanks to a discrete Lipschitz estimate on $\tilde E_N$. Using this idea, careful estimates allow to control the influence of one single particle. Then, the number of particles in the domain is bounded, again with the help of $\| f_N\|_\infty$ and $\Winf$. 
  
\medskip
All of this leads to the following estimate
\[
\frac{\tWinf(t) - \tWinf(t-\ep)}{\ep} \le   C \Bigl( \tWinf(t) + \ep^{\beta_1} \tWinf^d(t) + \ep^{\beta_2} \tWinf^{2d}(t)
\tilde\,  d_N(t)^{-\alpha} \Bigr),
\]
where $\beta_1, \beta_2 >0$ under the assumptions of Theorem~\ref{thm:deter}. The three terms of the r.h.s. come respectively form the three domains mentioned above. We complete the proof with an inequality on $\tilde d_N(t)$ obtained in a similar way $(\beta_3,\beta_4>0)$
\[
\frac{\tilde d_N(t) - \tilde d_N(t-\ep)}{\ep} \ge  
-  C \Bigl( \tilde d_N(t) + \ep^{\beta_3} \tWinf^d(t) + \ep^{\beta_4} \tWinf^{2d}(t) \tilde\,  d_N(t)^{-\alpha} \Bigr).
\]
The two previous inequalities form an (implicit) time discretization of an system of two differential inequalities. {As the non-linear terms come with small weight $\ep^{\beta_i}$, the previous system} provide uniform bounds until a critical time $T_\ep$ with $T_\ep\rightarrow \infty$ as $\ep\rightarrow 0$; hence for any fixed $T$, $T_\ep>T$  for $N$ large enough (depending on $T$).

\medskip
\paragraph{About the restriction $\alpha <1$.} 
This restriction is clearly manifested when two particles with non vanishing relative velocity becomes relatively close. The physical explanation is clear: if $\alpha<1$ the deviation in velocity due to a collision (another particle coming very close) is small. In particular there cannot be any fast variation in the velocities of the particles.  This is why it is enough to control the distance in $\R^{2d}$ between particles. In contrast when $\alpha>1$, a particle coming very close to another one can change its velocity over a very short time interval (even if their relative velocity remains of order $1$). Such ``collisions'' are incompatible with our argument, which requires a control on $\Winf$, \emph{i.e.} a control on all the trajectories.


%
\paragraph{The propagation of chaos results.} 
To deduce Theorem~\ref{thm:prob} from Theorem~\ref{thm:deter}, it is enough to show that the conditions $i)$ and $iii)$ under which our mean-field limit theorem is valid, are satisfied with large probability in the limit. This relies on already known results or on rather simple statistical estimates: 
\begin{itemize}
\item for point $i)$, it relies on a large deviation bound for $\| f_N\|_\infty$, See Proposition~\ref{prop:largedev}, 
\item for point $iii)$ it relies on a simple estimate (not of large deviation type) on $d_N(0)$ proved in~\cite{Hau09}, See Proposition~\ref{prop:dN},
\item and finally, we use also some large deviation bound on $W_1(\mu_N^0,f^0)$ obtained by Boissard~\cite{BoissardPhD}, see Proposition~\ref{probaint}.
\end{itemize}

\section{Proof of Theorem \ref{thm:deter} and \ref{thm:cutoff}}
%
\subsection{Definition of the transport}
We now try to compare the dynamics of $\mu_N$ and $f_N$,
which both have a compact support.  For that, we choose an optimal transport $T^0$ 
(of course depending on $N$) 
from $f_N^0$ to $\mu_N^0$ for the infinite MKW distance.
The existence of such a transport is ensured by
\cite{ChaDePJuu08}.  $T^0$ is defined on the support of $f_N^0$, which is
included in $B_{2d}(0, R^0)$ (the size of the support), and 
Proposition~\ref{prop:Winfbound} implies that
$W_\infty(f_N^0, \mu_N^0) \leq \ep$.  

\medskip
Thanks to the assumptions of both theorems, the  strong solution $f_N$ to the Vlasov equation is well defined till a time $T^\ast$, infinite in the case of Theorem~\ref{thm:deter}, that depends only on $C_\infty$ and $R_0$ and not on $N$.  Since we are dealing with strong solutions, there exists a well-defined underlying flow,  that we will denote by $Z^f =(X^f,V^f)$ : $Z^f(t,s,z)$ being the position-velocity at time $t$ of a particle with position-velocity $z$ at time $s$.

Moreover, by the assumption of Theorem~\ref{thm:deter} or because we use a cut-off in Theorem~\ref{thm:cutoff}, the dynamic of the $N$ particles is well defined, and we can also write in that case a flow $Z^\mu=(X^\mu,V^\mu)$, which is well defined at least at the position and velocity of the particles we are considering.
A simple way to get a transport of $f_N(t)$ on $\mu_N(t)$ is to transport along the flows the map $T^0$,
i.e. to define 
\[
T^t = Z^\mu(t,0) \circ T^0 \circ Z^f(0,t), \qquad  \text{and} \quad  T^t
=(T^t_x,T^t_v)
\]
We use the following notation, for a test-``particle'' of the
continuous system with position-velocity $z_t= (x_t,v_t)$ at time $t$, $z_s=
(x_s,v_s)$ 
will be its position ad velocity at time $s$ for $s \in [t-\tau,t]$. Precisely 
\[
z_s = Z^f(s,t,z_t)
\]
Since $f_N$ is the solution of a transport equation, we have
$f_N(t,z_t)=f_N(s,z_s)$. And since the vector-field of that transport equation
is divergence free, the flow $Z^f$ is measure-preserving in the sense that for 
all smooth test functions $\Phi$
\[
\int \Phi(z)\,f_N(s,z)\,dz=\int
\Phi(Z^f(s,t,z))\,f_N(t,z)\,dz=\int
\Phi(z_s)\,f_N(t,z_t)\,dz_t. 
\]
Finally, let us remark that the $f_N$ are solutions to the (continuous) Vlasov
equations with an initial $L^\infty$ norm and support that are
uniformly bounded in $N$.  Therefore the Proposition~\ref{prop:WPvlasov}, and in particular the last assertion in it imply that  this remains true uniformly in $N$
for any finite time $T < T^\ast$. In particular the uniform bound on the support 
$R(T)$ implies since $\alpha < d-1$ the existence of a constant C independent 
of $N$ such that for any $t\in [0,\ T]$
\begin{equation}\begin{split}
&\|f_N(t,.,.)\|_\infty\leq C,\qquad  \|f_N(t,.,.)\|_{L^1}=1,\\
&{\rm supp}\, f_N(t,.,.)\in B_{2d}(0,C),\\
& |E_{f_N}|_\infty(t) := \|E_{f_N}(t,\cdot)\|_\infty \leq \sup_x \int
|F(x-y)|\,f_N(t,y,w)\,dy\,dw\leq C  \\
&|\nabla E_{f_N}(t,x)|\leq \int |\nabla F(x-y)|\,f_N(t,y,w)\,dy\,dw \leq C.
\end{split}\label{boundfN}
\end{equation}

\medskip

In what follows, the final time $T$ is fixed and independent of
$N$. For simplicity, $C$ will denote a generic universal constant,
which may actually depend on $T$, the size of the initial support,
the infinite norms of the $f_N$... But those constants are always independent of
$N$ as in \eqref{boundfN}.

\subsection{The quantities to control}

We will not be able to control the infinite norm of the field (and its
derivative) created by the empirical distribution $\mu_N$, but only a small
temporal average of this norm. For this, we introduce in the case
without cut-off a small time 
step $\tau=\ep^{r'}$ for some $r'>r$ and close to $r$  (the precise condition
will appear later).
In the case with  cut-off where $r$ and $r'$ are useless, the time step will
by $\tau = \ep$.

\medskip
Before going on, we define some important quantities :

\begin{itemize}
\item {\bf The MKW infinite distance between $\mu_N(t)$ and $f(t)$.}

We wish to bound the infinite Wasserstein distance
$\Winf(\mu_N(t),f_N(t))$ between the empirical measure $\mu_N$ associated to 
the $N$
particle system~\eqref{eq:ODE}, and the solution $f_N$ of the Vlasov
equation~\eqref{eq:vlasov} with blobs as initial condition. But for 
convenience we will work instead with the quantity
\begin{equation} \label{def:Winfty}
\Winf(t) :=  \sup_{s \leq t} \sup_{z_s\in {\rm supp}\,f_N(s)}
|T^s(z_s)-(z_s)|,
\end{equation}
where the $\sup$ on $z_s$ should be understood precisely as a essential
supremum with respect to the measure $f_N(s)$. This is not exactly the
infinite Wasserstein distance between $\mu_N(t)$ and $f_N(t)$ (or its
supremum in times smaller than $s \le t$). But, since for all $s$, the
transport map $T^s$ send the measure $f_N$ onto $\mu_N$ by construction, we
always have
$$
\Winf(\mu_N(t),f_N(t)) \le \sup_{s \leq t} \Winf(\mu_N(t),f_N(t))
\le \Winf(t).
$$
So that a control on $\Winf(t)$ implies a control on $\Winf(\mu_N(t),f_N(t))$.
It is in fact a little stronger, since it means that
rearrangements in the transport are not necessary to keep the infinite MKW
distance bounded.
We introduce the supremum in time for technical reasons as it will be simpler
to deal with a non decreasing quantity in the sequel.

\item {\bf The support of $\mu_N$}

We also need a uniform control on the support in position and
velocity of the empirical distributions : 
\begin{equation} \label{def:RN}
R^N(t):= \sup_{s \le t } \max_i |(X_i(t),\;V_i(t))|.
\end{equation}

\item  {\bf The infinite norm $\dENinf$ of the time averaged discrete
  derivative of the force field} 

We define  a version of the infinite norm of the averaged derivative
of the discrete force field $E_N$
\begin{equation} \label{def:dEinfty}
\dENinf(t) := \sup_{i \neq j}\frac1\tau\int_{t-\tau}^t \frac{|E_N(X_i(s))
  -E_N(X_j(s))| \,ds }{ |X_i(s) -X_j(s)| + \ep^{(1+r')}}\;ds. 
\end{equation}
For 
$\dENinf$, we use
the convention that when the interval of integration contains $0$ (for $t <
\tau$), the integrand is null on the right side for negative times.
Remark that the control on $\dENinf$ is useless in the cut-off case.

\item {\bf The minimal distance in $\R^{2d}$, $d_N$}

 which has already be defined by the equation \eqref{eq:dmin} in the Section
\ref{sec:main}.
 
\item {\bf Two useful integrals $\Ialp(t,\bz_t,z_t)$ and $\Jalp(t,\bz_t,z_t)$} 
 
Finally for any two test trajectories $z_t$ and $\bz_t$, we define 
\begin{equation} \label{def:Ialpha}
\Ialp(t,\bz_t,z_t) := \frac1\tau \int_{t-\tau}^t |F(T^s_x(\bz_s) -
  T^s_x(z_s)) -F(\bx_s - x_s)|   \,ds,
\end{equation}
which controls the difference of the two force fields at two point
related by the ``optimal'' transport. We recall that we use here the convention
$F(0)=0$, in order to avoid self-interaction. It is important here since we
have $T^s(z_s) = T^s(\bar z_s)$ for all $s \in [t-\tau,t]$, for a set of
$(z_s,\bar z_s)$ of positive measure (those who are associated to the same 
particle $(X_i,V_i)$).

Defining a second kernel as
\begin{equation} \label{def:Kep}
K_\ep :=\min\left(\frac1{|x|^{1+\alpha}},\;\frac1{\ep^{1+r'}\,
|x|^{\alpha}}\right) \quad \text{for } x \neq 0, \quad \text{and }\; 
K_\ep(0) = 0,
\end{equation}
we introduce a second useful quantity
\begin{equation} \label{def:Jalpha}
\begin{split}
\Jalp(t,\bz_t,z_t) & :=  \frac1\tau \int_{t-\tau}^t K_\ep(|T^s_x(\bz_s) -
T^s_x(z_s) | \,ds   \\
& = \frac1\tau \int_{t-\tau}^t  K_\ep(|X_i(s) -
X_j(s) | \,ds,
\end{split} 
\end{equation}
if $i$ and $j$ is the indices such that $Z_i(t)= T^t(\bz_t)$ and $Z_j(t) =
T^t(z_t)$. $J_{\alpha+1}$ will be useful to control the discrete
derivative of the field $\dENinf(t)$, and is thus useless in the cut-off case.
\end{itemize}

All previous quantities are relatively easily bounded by $\Ialp$ and $\Jalp$. 
Those last two will not be bounded 
by direct calculation on the discrete system, but we will compare them to
similar ones for the continuous system, paying for that in terms of
the distance between $\mu_N(t)$ and $f(t)$.
That strategy
is interesting because the integrals are easier to manipulate than
the discrete sums.

\begin{remark} \label{rem:time0}
Before stating the next Proposition, let us mentioned that we also define for $t
<0$,  $W(t)=W(0)$ and $d_N(t) = d_N(0)$. This is just a helpful convention. With
it the estimate of the next Proposition are valid for any $t \ge 0$, and this
will be very convenient in the conclusion of the proof of our main theorem. Remark
also that $\dENinf(0)=0$. 
\end{remark}

We summarize the first easy bounds in the following
\begin{prop}
Under the assumptions of Theorem~\ref{thm:deter}, one has 
for some constant
$C$ uniform in $N$, that for all $t  \ge 0$
\begin{align*} \label{eq:contbound}
(i) \quad  & R_N(t)\leq \Winf(t)+R(t)\leq \Winf(t)+ C ,\\
(ii) \quad & \Winf(t)\leq \Winf(t-\tau) + \tau \Winf(t) +C\,\tau \sup_{\bz_t} 
\int_{|z_t|\leq R(t)} \Ialp(t,\bz_t,z_t)\,dz_t,\\
(iii) \quad &\dENinf(t) \leq C\,\sup_{\bz_t} \int_{|z_t|\leq R(t)}
\Jalp(t,\bz_t,z_t)\,dz_t, \\
(iv) \quad & d_{N}(t) + \ep^{1+r'} \geq [d_{N}(t- \tau) + \ep^{1+r'}]
e^{-\tau (1+ \dENinf(t))}.
\end{align*}
The points $i)$ and $ii)$ are also 
satisfied under the assumptions of~\ref{thm:cutoff}.
\label{propeasy}
\end{prop}

Note that the control on $R_N(t)$ is simple enough that it will
actually be used implicitly in the rest many times, and that the $iv)$ is a
simple consequence of the $iii)$. In fact, in that
proposition the crucial estimates are the $ii)$ and $iii)$.
Remark also that in the case of very singular interaction force
($\alpha\geq 1$) with cut-off - in short 
$(S^\alpha_m)$  conditions~\eqref{eq:Ckappa'} - the control on
minimal distance $d_N$ and therefore the control on $\dENinf$ are useless,
so that the only interesting inequality is the second one.
%
\subsection{Proof of Prop. \ref{propeasy}}

{\sl Step 1. } Let us start with $(i)$. Simply write
\[
R^N(t)= \sup_{s \le t} \sup_{z_s \in {\rm supp}\,f_N(s,\cdot)} |T^s(z_s)|\leq
\sup_{s \le t} \sup_{z_s \in {\rm supp}\,f_N(s,\cdot)} |T^s(z_s)-z_t|+
\sup_{s \le t} \sup_{z_s \in {\rm supp}\,f_N(s,\cdot)} |z_s|,
\]
So indeed by the bound~\eqref{boundfN} and the definition~\eqref{def:Winfty} of
$\Winf$
\[
R^N(t)\leq 
\Winf(t)+ C. 
\]

\bigskip

{\sl Step 2. } For $(ii)$, for any time $t' \in [t-\tau,t]$ we have
\begin{align}
| T^{t'}_x(\bar z_{t'}) - \bar x_{t'} | & \le  | T^{t-\tau}_x(\bar z_{t-\tau}) -
\bar x_{t-\tau} | + \int_{t-\tau}^{t'} | T^{s}_v(\bar z_s) - \bar v_s | \,ds 
\nonumber \\
& \le | T^{t-\tau}_x(\bar z_{t-\tau}) - \bar x_{t-\tau} | + \tau \Winf(t),
\label{estim:Ttx}
\end{align}
and for the speeds
\begin{align*}
| T^{t'}_v(\bar z_{t'}) - \bar v_{t'} | & \le  | T^{t-\tau}_v(\bar z_{t-\tau}) -
\bar v_{t-\tau} |\\
&\qquad + \int_{t-\tau}^{t'}  \int | F(T^s_x(\bar z_s) - T^s( z_s)) -
F(\bar x_s - x_s) | f_N(s,z_s)\,dz_s ds \\
 \le &| T^{t-\tau}_v(\bar z_{t-\tau}) -
\bar v_{t-\tau} |  +
\int_{t-\tau}^t  \int | F(T^s_x(\bar z_s) - T^s( z_s)) -
F(\bar x_s - x_s) | f_N(t,z_t)\,dz_t ds.
\end{align*}
where we used the fact that the change of variable $z_t \mapsto z_s$ preserves
the measure. 
Since $f_N$ is uniformly bounded in $L^\infty$ and compactly supported
in $B(0,R(t))$, one gets by the definition~\eqref{def:Ialpha} of $\Ialp$
\begin{equation} \label{estim:Ttv}
| T^{t'}_v(\bar z_{t'}) - \bar v_{t'} | \le | T^{t-\tau}_v(\bar z_{t-\tau}) -
\bar v_{t-\tau} |  + C \tau 
\sup_{\bz_t}\int_{|z_t|\leq R(t)}
\Ialp(t,\bz_t,z_t)\,dz_t.
\end{equation}
Summing the two estimates~\eqref{estim:Ttx} and~\eqref{estim:Ttv},  we get for the Euclidean distance on $\R^{2d}$
\[
| T^{t'}(\bar z_{t'}) - \bar z_{t'} | \le   | T^{t-\tau}(\bar z_{t-\tau}) -
\bar z_{t-\tau} | + C \tau \biggl(   \Winf(t)+ 
\sup_{\bz_t}\int_{|z_t|\leq R(t)} \Ialp(t,\bz_t,z_t)\,dz_t  \biggr).
\]
Taking
the supremum over all $\bz_{t'}$ in the support of $f_N(t')$, and then the
supremum over all $t' \in [t-\tau,t]$ we get
$$
\Winf(t) \le \Winf(t-\tau) + \tau \Winf(t) + C \tau \sup_{\bz_t}\int_{|z_t|\leq
R(t)}
\Ialp(t,\bz_t,z_t)\,dz_t
$$
which is exactly $(ii)$.

\bigskip
{\sl Step 3. }
Concerning  $|\nabla^N E|_\infty(t)$ in $(iii)$, noting that \[\begin{split}
\inttau
\frac{|E_N(X_i(s)) -E_N(X_j(s))|}{|X_i(s) -X_j(s)| + \ep^{1+r'}} = 
\frac1N  & \sum_{k \neq i,j}    \inttau
\frac{|F(X_i(s) - X_k(s)) - F(X_j(s) - X_k(s))|}{|X_i(s) -X_j(s)| +
\ep^{1+r'}} \, ds .   \\
& + \frac1N \inttau
\frac{|F(X_i(s) - X_j(s)) - F(X_j(s) - X_i(s))|}{|X_i(s) -X_j(s)| +
\ep^{1+r'}} \, ds 
\end{split}\] 
By the assumption \eqref{eq:Calpha}, one has that
\[
|F(x)-F(y)|\leq C \left(\frac1{|x|^{\alpha+1}}+\frac1{|y|^{\alpha+1}} \right)
|x-y|.
\]
So
\[
\frac{|F(X_i(s) - X_k(s)) - F(X_j(s) - X_k(s))|}{|X_i(s) -X_j(s)| +
\ep^{1+r'}}\leq \frac{C}{|X_i(s) - X_k(s)|^{1+\alpha}}+
\frac{C}{|X_j(s) - X_k(s)|^{1+\alpha}} ,
\]
and that bound is also true for the remaining
term where $k=i$ or $j$, if we delete the undefined term in the sum.
One also obviously has, still by \eqref{eq:Calpha}
\[\begin{split}
\frac{|F(X_i(s) - X_k(s)) - F(X_j(s) - X_k(s))|}{|X_i(s) -X_j(s)| +
\ep^{1+r'}}\leq  &\frac{C}{\ep^{1+r'}|X_i(s) - X_k(s)|^{\alpha}}\\
&+\frac{C}{\ep^{1+r'}|X_j(s) - X_k(s)|^{\alpha}}.
\end{split}\]
Therefore by the definition of $K_\ep$
\[
\frac{|F(X_i(s) - X_k(s)) - F(X_j(s) - X_k(s))|}{|X_i(s) -X_j(s)| +
\ep^{1+r'}}\leq C  \bigl[ K_\ep(X_i(s) - X_k(s))+K_\ep(X_j(s) - X_k(s)) \bigr].
\]
Summing up, this implies that
\[\begin{split}
|\nabla^N E|_\infty(t) \leq C \max_{i\neq j}  \Big( &\frac1\tau \inttau
\frac1N  \sum_{k \neq i} K_\ep(X_i(s) - X_k(s))  \,ds \\
&+   \frac1\tau
\inttau  \frac1N \sum_{k \neq
j} K_\ep(X_j(s) - X_k(s)) ds \Big) . 
\end{split}\]
Transforming the sum into integral thank to the transport, we get exactly the
bound $(iii)$ involving $\Jalp$.

\bigskip
{\sl Step 4. }
Finally for $d_N(t)$, consider any $i\neq j$, differentiating the Euclidean distance $|Z_i-Z_j|$, we get
\[
\frac{d}{ds}|(X_i(s)-X_j(s),\;V_i(s)-V_j(s))|\geq
-|V_i(s)-V_j(s)|-|E_N(X_i(s))-E_N(X_j(s))|. 
\]
Simply write
\[
|E_N(X_i(s))-E_N(X_j(s))|\leq
\frac{|E_N(X_i(s))-E_N(X_j(s))|}{|X_i(s)-X_j(s)|+\ep^{1+r'}}\;
(|X_i(s)-X_j(s)|+\ep^{1+r'})
\]
to obtain that
\[\begin{split}
\frac{d}{ds}|(X_i(s)-X_j(s),\;V_i(s)-V_j(s))|\geq&
-\left(1+\frac{|E_N(X_i(s))-E_N(X_j(s))|}
{|X_i(s)-X_j(s)|+\ep^{1+r'}}\right)\\
&\quad (|(X_i(s)-X_j(s),\;V_i(s)-V_j(s))|+\ep^{1+r'}). 
\end{split}\]
Integrating this inequality and taking the minimum, we get
\[\begin{split}
d_N(t) +  \ep^{1+r'} & \geq (d_N(t-\tau)+\ep^{1+r'})  \,\inf_{i\neq j}\,
\exp\left(-\tau - \inttau\frac{|E_N(X_i(s))-E_N(X_j(s))|}
{|X_i(s)-X_j(s)|+\ep^{1+r'}}\,ds \right)\\
& \geq  [d_{N}(t- \tau) + \ep^{1+r'}] \exp^{-\tau(1 + \dENinf(t))}.
\end{split}\]

\subsection{The bounds for $\Ialp$ and $\Jalp$}
%
To close the the system of inequalities in Proposition \ref{propeasy}, it
remains to bound the two integrals involving $I_\alpha$ and $J_\alpha$. It is
done with the following lemmas

\begin{lemma} \label{lem:boundI}
Assume that $F$ satisfies an $(S^\alpha)$-condition~\eqref{eq:Calpha} with 
$\alpha<1$, and that
$\tau$ is small enough such that for some constant $C$ (precise in the proof)
\begin{equation} \label{cond:lem1}
 C \, \tau\,(1+|\nabla^N
E|_\infty(t))\,(\Winf(t)+\tau)\leq d_N(t).
\end{equation}
Then one has the following bounds, uniform in $\bz_t$
\[\begin{split}
\int_{|z_t| \leq R(t)} \Ialp(t,\bz_t,z_t)\,dz_t\leq C\, \big[
\Winf(t) + (\Winf(t)+\tau)^d\tau^{-\alpha} +(\Winf(t)+\tau)^{2d}
(d_N(t))^{-\alpha}  \tau^{-\alpha}\big].
\end{split}\]

\[\begin{split}
\int_{|z_t \leq R(t)} \Jalp(t,\bz_t,z_t)\,dz_t\leq C\,&\big(
1 
+(\Winf(t)+\tau)^d\ep^{-(1+r')}\,\tau^{-\alpha}\\
&+(\Winf(t)+\tau)^{2d}
\ep^{-(1+r')}\,\tau^{-\alpha}\,(d_N(t))^{-\alpha}\big).
\end{split}\]
\label{boundIK}
\end{lemma}

\bigskip
In the cut-off case where the interaction force satisfy a
$(S^\alpha_m)$ condition~\eqref{eq:Ckappa'},
 we only need to bound the integral of
$I_\alpha$, with the result  
\begin{lemma} \label{boundIcut}
Assume that  $1 \leq \alpha<d-1$, and that $F$ satisfies a $(S^\alpha_m)$
condition~\eqref{eq:Ckappa'}. Then one as the following bound, uniform in 
$\bz_t$
\begin{equation}
\int_{|z_t| \leq R(t)} \Ialp(t,\bz_t,z_t)\,dz_t\leq C\, \big(
\Winf(t) + (\Winf(t)+\tau)^d \tau^{-1} \ep^{\bar m (1-\alpha)}
+(\Winf(t)+\tau)^{2d}
\ep^{- \bar m \alpha}\big).
\label{boundIKcut}
\end{equation}
with the convention\footnote{That convention may be justified by the fact that
it implies
  a very simple algebra $(x^{1-\alpha})'
\approx x^{-\alpha}$ even if $\alpha=1$. It allows us to give an
unique formula rather than three different cases.} (if $\alpha =1$) that $\ep^0
= 1+  |\ln \ep|$
.
\end{lemma}

\bigskip
The proofs with or without cut-off follow the same line and we will prove the
above lemmas at the same time. We begin by an explanation of the sketch of
the proof, and then perform the technical calculation.
%
%
\subsubsection{Rough sketch of the proof}
%
%
The point $\bz_t=(\bx_t,\bv_t)$ is considered fixed through all this subsection
(as the integration is carried over $z_t=(x_t,v_t)$). Accordingly we
decompose the integration in $z_t$ over several domains. First
\begin{equation} \label{def:At}
A_t = \{ z_t\, | \; |\bx_t - x_t| \geq  4 \Winf(t) + 2 \tau( |\bv_t
-v_t| + \tau |E|_\infty(t)) \,\}.
\end{equation}

This set consist of points $z_t$ such that $x_s$ and $T^s_x(z_s)$ are
sufficiently far away from $\bx_s$ on the whole interval $[t-\tau,t]$, so that
they will not see the singularity of the force. The bound over this domain will
be obtained using traditional estimates for convolutions.

Next, one part of the integral can be estimated easily on $A_t^c$ (the part
corresponding to the flow of the regular solution $f_N$ to the Vlasov
equation). For the other part it is necessary to decompose further.
The next domain is
\begin{equation} \label{def:Bt}
B_t=A_t^c\bigcap 
\{z_t\,|\; |\bv_t-v_t|\geq 4\,W_\infty(t)+4\,\tau |E|_\infty|(t)\}.
\end{equation}

This contains all particles $z_t$ that are close to $\bz_t$ in position
(i.e. $x_t$ close to $\bx_t$), but with enough relative velocity not to
interact too much. The small average in time will be useful in that part, as
the two particles remains close only a small amount of time.

The last part is of course the remainder
\begin{equation} \label{def:Ct}
C_t=(A_t\cup B_t)^c.
\end{equation}
This is a small set, but where the particles remains close together a
relatively long time. Here, we are forced to deal with the corresponding
term at the discrete
level of the particles. This is the only term which requires the
minimal distance in $\R^{2d}$; and the only term for which we need a
time step $\tau$ small enough as per the assumption in Lemma
\ref{boundIK}.
\begin{figure}[ht]
  \begin{center}
  \scalebox{.7}{ 
    \includegraphics{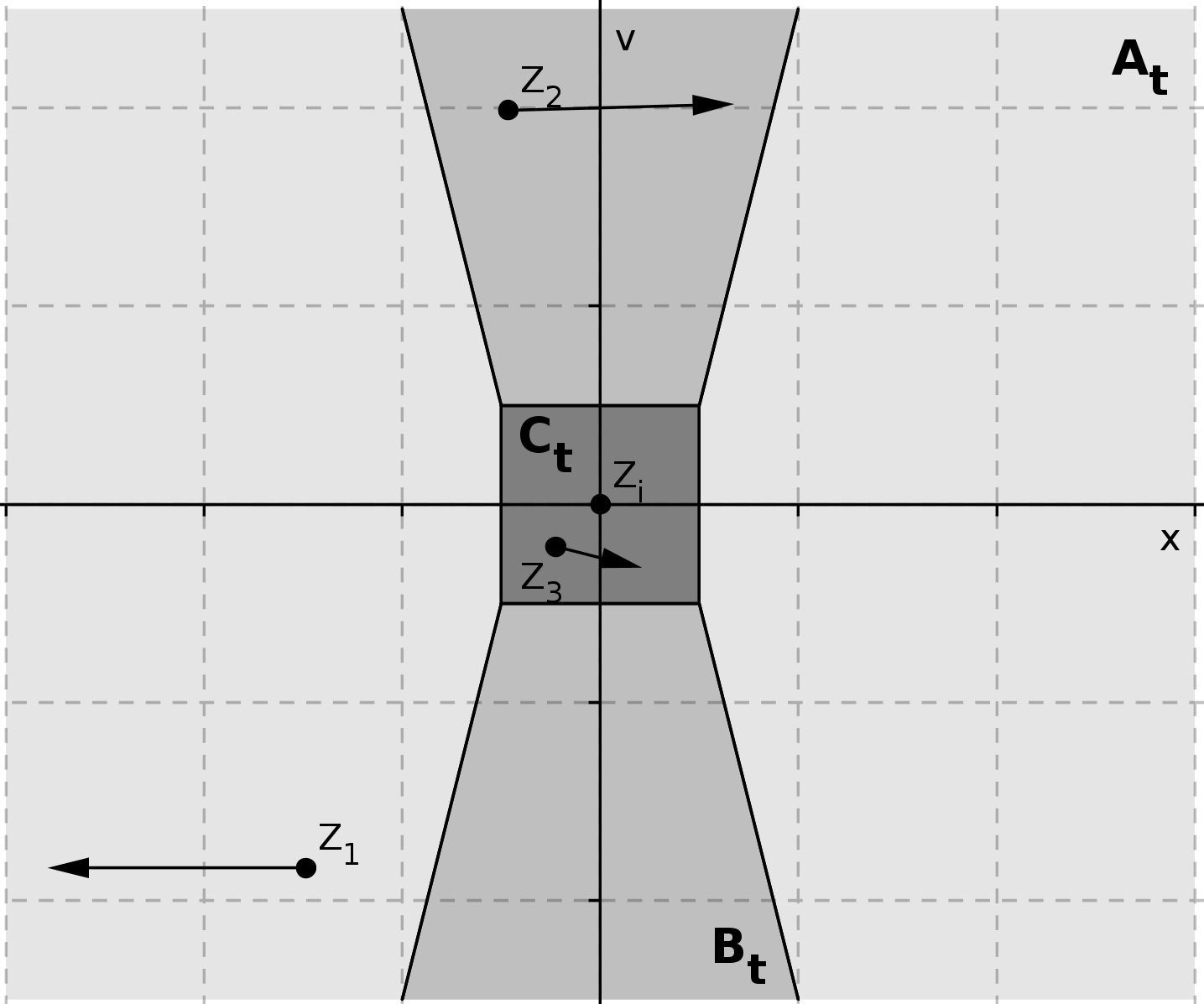}
  }
  \caption{The partition of $\R^{2d}$.}
  \end{center}
\end{figure}

%
\subsubsection{Step 1: Estimate over $A_t$}
%
According to the definition~\eqref{def:At}, if $z_t \in A_t$, we have for  $s 
\in [t-\tau,t]$
\begin{eqnarray}
|\bx_s - x_s |  & \geq & |\bx_t-x_t| - (t-s)|\bv_t-v_t| - (t-s)^2
|E|_\infty(t) \geq  \frac{|\bx_t-x_t|}2  \\ 
|T^s_x(\bz_s) - T^s_x(z_s)| & \geq & |\bx_s - x_s |  -2 \Winf(s)
\geq \frac{|\bx_t-x_t|}2. \label{eq:FAR}
\end{eqnarray}

For $\Ialp$, we use the direct bound for $z_t \in A_t$
\begin{eqnarray*}
|F(T^s_x(\bz_s) - T^s_x(z_s)) -F(\bx_s - x_s)| & \leq &
\frac C{|\bx_t-x_t|^{1+\alpha}} (|T^s_x(\bz_s) -\bx_s| +
|T^s_x(z_s)-x_s|) \\ 
& \leq & \frac C {|\bx_t-x_t|^{1+\alpha}} \Winf(s) \leq
\frac{C}{|\bx_t-x_t|^{1+\alpha}} \Winf(t), 
\end{eqnarray*}
and obtain by integration on $[t-\tau,t]$
\[
\Ialp(t,\bz_t,z_t) \leq  \frac{C}{|\bx_t-x_t|^{1+\alpha}} \Winf(t).
\]
Then integrating in $z_t$ we may get since $\alpha +1 < d $ 
\begin{equation}\begin{split}
\int_{A_t}   \Ialp(t,\bz_t,z_t)  \,dz_t  & \leq  C 
\Winf(t) \int_{A_t} 
\frac{dz_t}{|\bx_t-x_t|^{1+\alpha}} 
\\
& \leq  C\,R(t)^{2d-1-\alpha}\,\Winf(t) \leq C\,\Winf(t)  \label{At}.
\end{split}\end{equation}

For $\Jalp$, we use \eqref{eq:FAR} on the set $A_t$ the bound
$$ | K_\ep(T^s_x(\bz_s) - T^s_x(z_s)) | \leq \frac C {|\bx_t -x_t|^{1+ \alpha}}.
$$
Integrating with respect to time and $z_t$ we get since $1+ \alpha < d$.
\begin{equation}\begin{split}
\int_{A_t}   \Jalp(t,\bz_t,z_t)  \,dz_t  & \leq  C 
 \int_{A_t} 
\frac{dz_t}{|\bx_t-x_t|^{1+\alpha}} 
\\
& \leq  C\,R(t)^{2d-1-\alpha} \leq C  \label{At2}.
\end{split}\end{equation}

 For the cut-off case, the estimation on $\Ialp$ for this step is unchanged.

\subsubsection{Step 1' : Estimate over $A_t^c$ for the ``continuous'' part
  of $\Ialp$ 
  .}
%
For the remaining term in $\Ialp$, we use the rude bound
$$ |F(T^s_x(\bz_s) - T^s_x(z_s)) -F(\bx_s - x_s)| \leq |F(T^s_x(\bz_s) -
T^s_x(z_s))|+| F(\bx_s - x_s)|. $$
The term involving $T^s$  is complicated and requires the additional
decompositions. It will be treated in the next sections.
The other term is simply bounded by
\[ \begin{split}
  \int_{z_t \in A_t^c } \frac1 \tau \inttau  |F(\bx_s -x_s)| ds\,\,dz_t   &
\leq 
\frac1 \tau \inttau \int_{z_t \in A_t^c }  \frac {C \,dz_t}{|\bx_s
-x_s|^{\alpha}} \, ds \\
& \leq  \frac1 \tau \inttau \int_{z_s \in Z^f(s,t,A_t^c) }  \frac {C
\,dz_s}{|\bx_s
-x_s|^{\alpha}} \, ds.
\end{split} \]
From the bounds~\eqref{boundfN}, 
we get that
\[
|A_t^c| \leq C R(t)^d [W_\infty(t) + \tau( 1+ |E|_\infty(t)) ]^d \leq C
(W_\infty(t) + \tau)^d,
\]
where $|\cdot|$ denote the Lebesgue measure.
Since the flow $Z_f$ is measure preserving, the measure of the set
$Z^f(s,t,A_t^c)$ satisfies the same bound. This set is also included in
$B_{2d}(0,R)$. We use the above lemma which implies that above all the
set $Z(s,t,A_t^c)$, the integral reaches is maximum when the set is a cylinder

\begin{lemma} Let $\Omega\subset B_{2d}(0,R)\subset \R^{2d}$. Then for any $a<d$,
there exists a constant $C_a$ depending on $a$ and $d$ such that
\[
\int_\Omega \frac{dz}{|x|^a}\leq C_a R^a\,|\Omega|^{1-a/d}.
\]\label{cylinder}
\end{lemma}

\noindent{\bf Proof of Lemma \ref{cylinder}.} We maximize 
the integral 
\[\int_{\omega} |x|^{-a} dz\]
over all sets $\omega\subset\R^{2d}$ 
satisfying $\omega\subset B_{2d}(0,R)$ and $|\omega|=|\Omega|$. 
It is clear that the maximum is obtained by concentrating as much as
possible $\omega$ near $x=0$, {\em i.e.} with a
cylinder of the form $B_d(0,r) \times B_d(0,R)$.
Since $|\omega|=|\Omega|$ we have
$(c_d)^2 r^d R^d =|\Omega|$, where $c_d$ is the volume of the unit ball of
dimension $d$. The integral over this cylinder can now be computed
explicitly and gives the lemma.\cqfd

\medskip
Applying the lemma, we get
\begin{equation} \label{Ic}
 \int_{z_t \in A_t^c } \frac1 \tau \inttau  |F(\bx_s -x_s)| \,dz_t ds  \leq C
 [\Winf(t)+\tau]^{d-\alpha}.
\end{equation}

That term do not appear in Lemma \ref{lem:boundI} since it is
strictly smaller than the bound of the remaining term (involving $T$), as we
will see in the next section.

For the cut-off case, the same bound is valid for $\Ialp$ since $\alpha \leq d-1
< d$ (The cut-off cannot in fact help to provide a better bound for this term).

\medskip

At this point, the remaining term to bound in $\Ialp$ is only
\begin{equation} \label{eq:remain}
\int_{z_t \in A_t^c } \frac1 \tau \inttau |F(T^s_x(\bz_s) -
T^s_x(z_s))| \,ds
\end{equation} 
and the remainder in $\Jalp$ is 
\eqref{eq:remain}
\begin{equation}
\int_{A_t^c} \Jalp(t,\bz_t,z_t)\,dz_t = 
\frac1 \tau  \int_{A_t^c } \inttau 
{K_\ep(T^s_x(\bz_s) - T^s_x(z_s))} \,dz_t ds.\label{remainderJalp}
\end{equation}
Therefore in the next sections we focus on giving a bound for~\eqref{eq:remain}
and~\eqref{remainderJalp}.
%
%
\subsubsection{Step 2: Estimate over $B_t$}
%
%
We recall the definition of $B_t$
\[
B_t = \left\{ z_t\; \text{ s.t. }
\begin{array}{l}
 |\bx_t - x_t| \leq  4 \Winf(t) + 2 \tau( |\bv_t
-v_t| + \tau |E|_\infty(t)) \, \\
|\bv_t-v_t|\geq 4\,W_\infty(t)+4\,\tau |E|_\infty|(t)
\end{array}\right\}. 
\]

If $z_t \in B_t$, we have for  $s \in [t-\tau,t]$
\begin{eqnarray} 
|\bv_s - v_s -\bv_t+v_t|  & \leq & 2\tau|E|_\infty(t) 
\leq  \frac{|\bv_t-v_t|}2,  \\ 
|T^s_v(\bz_s) - T^s_v(z_s)-\bv_t+v_t| & \leq & |\bv_s - v_s-\bv_t+v_t|  
+2 \Winf(s) \leq \frac{|\bv_t-v_t|}2.  \label{eq:CBF}
\end{eqnarray}
This  means that the particles involved are close
to each others (in the positions variables), but with a sufficiently
large relative velocity, so that they do not interact a lot on the interval
$[t-\tau,t]$.

First we introduce a notation for the term of  \eqref{eq:remain} 
\begin{equation} \label{eq:Ibc}
\int_{z_t \in B_t} I_{bc}(t,\bz_t, z_t) \,dz_t , \quad \text{with } 
I_{bc}(t,\bz_t,z_t) = I_{bc}(t,i,j) :=  \frac1\tau \inttau
F(T^s_x(\bz_s) -T^s_x(z_s)) \,ds,
\end{equation}
where $(i,j)$ are s.t. $T^s_x(\bz_s)=X_i(s)$, $T^s_x(z_s)=X_j(s)$.
For $z_t\in B_t$, define for $s\in [t-\tau,\ t]$
\[
\phi(s) : =(T_x^s(\bz_s)-T_x^s(z_s))\cdot
\frac{\bv_t-v_t}{|\bv_t-v_t|} = (X_i(s) - X_j(s)) 
\cdot \frac{\bv_t-v_t}{|\bv_t-v_t|}.
\] 
Note that $|\phi(s)|\leq |T_x^s(\bz_s)-T_x^s(z_s)|$ and that
\[\begin{split}
\phi'(s)&=(T_v^s(\bz_s)-T_v^s(z_s))\cdot
\frac{\bv_t-v_t}{|\bv_t-v_t|}\\
&=|\bv_t-v_t|+(T_v^s(\bz_s)-T_v^s(z_s)-(\bv_t-v_t))
\cdot
\frac{\bv_t-v_t}{|\bv_t-v_t|}
\geq \frac{|\bv_t-v_t|}2,
\end{split}\]
where we have used \eqref{eq:CBF}. Therefore $\phi$ is an increasing
function of the time on the interval $[t-\tau,t]$. If it vanishes at some time
$s_0 \in [t-\tau,t]$, then the previous bound by below on its derivative
implies that
\begin{equation} \label{eq:dispB}
|T_x^s(\bz_s)-T_x^s(z_s)| \ge |\phi(s)| \ge |t-s_0|\,\frac{|\bv_t-v_t|}2.
\end{equation}
If $\phi$ is always positive (resp. negative) on $[t-\tau,t]$, then the previous
estimate is still true with the choice $s_0 = t -\tau$ (resp. $s_0 = t$). So in
any case, estimate~\eqref{eq:dispB} holds true for some $s_0 \in [t-\tau,t]$.
Using this directly gives, as $\alpha<1$
\begin{equation} \label{vtrick}
| I_{bc}(t,\bz_t,z_t)|\leq  \frac{C}{\tau}\,|\bv_t-v_t|^{-\alpha}
 \int_{t-\tau}^t \frac{ds}{|s-s_0|^\alpha}\leq
C\,\tau^{-\alpha}\,|\bv_t-v_t|^{-\alpha}. 
\end{equation}
Now integrating
\[\begin{split}
\int_{z_t \in B_t} |I_{bc}(t,\bz_t,z_t)|\,dz_t & \leq
C\,\tau^{-\alpha}\int_{A_t^c}
\frac{dz_t}{|\bv_t-v_t|^\alpha}\\
&\leq C \,\tau^{-\alpha}\, [\Winf(t)+\tau]^d\,[R(t)]^{d-\alpha},
\end{split}\]
by using the fact that $B_t \subset B(0,C[\Winf(t)+\tau]) \times B(0,R(t))$. In conclusion
\begin{equation} \label{Bt}
\int_{z_t \in B_t} |I_{bc}(t,\bz_t,z_t)|\,dz_t \leq C \,\tau^{-\alpha}\,
[\Winf(t)+\tau]^d.
\end{equation}
With the  cut-off where $\alpha > 1$, the reasoning follows
the same line up to the bound \eqref{vtrick} which relies on the assumption
$\alpha <1$. \eqref{vtrick} is replaced by
\begin{eqnarray*}
|I_{bc}(t,\bz_t,z_t)| & \leq & \frac C \tau\,
 \inttau \frac{ds}{(|s-s_0| |\bv_t-v_t|+ 4 \ep^{\bar m})^\alpha} \\
 & \le & \frac C \tau\,
 \int_{t-\tau}^{s_0}  \ldots + \frac C \tau\,  \int_{s_0}^t  \ldots  
 \le \frac {2C} \tau\, \int_0^\tau \frac{ds}{( s|\bv_t-v_t|+ 4 \ep^{\bar
m})^\alpha} \\
& \leq & \frac{C}{\tau |\bv_t-v_t| }\, \int_0^{\tau|\bv_t-v_t|} \frac{ds}{(s +
4 \ep^{\bar m})^\alpha} \leq \frac{C \ep^{ \bar m (1 -
\alpha)}}{\tau |\bv_t-v_t|}.
 \end{eqnarray*}

When $\alpha =1$, the previous calculation leads to
\[
|I_{bc}(t,\bz_t,z_t)| \leq  \frac C{\tau |\bv_t-v_t| }\, \ln \left( 1 + C
\tau \ep^{- \bar m} \right) 
 \leq  \frac{C}{\tau |\bv_t-v_t| }\, ( 1 + \ln \ep^{1 - \bar m} )
\le  \frac{C \ep^0}{\tau }
\]
where the
second bound follows from $\ln(1+x) \le 1 + \ln(x)$ if 
$x \ge
1$.
 In the third one, we use that $\tau=\ep$ in the cut-off case, and in the last 
one, we use the convention $\ep^0 =  1+ |\ln (\ep)|$.

In both cases, the singular part in $1/|\bv_t -v_t|$ is integrable on
$\R^d$ and integrating that bound 
over $B_t$, we get the estimate
\begin{equation}
\begin{split}
\int_{z_t \in B_t} |I_{bc}(t,\bz_t,z_t)|\,dz_t & \leq
C\,\tau^{-1} \ep^{ \bar m (1 - \alpha)} \int_{A_t^c}
\frac{dz_t}{|\bv_t-v_t|}\\
&\leq C \,\tau^{-1}\, \ep^{ \bar m (1 - \alpha)}
[\Winf(t)+\tau]^d\,[R(t)]^{d-1}, \\
&\leq C \,\tau^{-1}\, \ep^{ \bar m (1 - \alpha)} [\Winf(t)+\tau]^d\,
\end{split}
\label{Btcutoff}
\end{equation}
%
%
\subsubsection{Step 3: Estimate over $C_t$}
%
%
We recall the definition of $C_t$
\[
C_t = \left\{ z_t\; \text{ s.t. }
\begin{array}{l}
 |\bx_t - x_t| \leq  4 \Winf(t) + 2 \tau( |\bv_t
-v_t| + \tau |E|_\infty(t)) \, \\
|\bv_t-v_t|\leq 4\,W_\infty(t)+4\,\tau |E|_\infty|(t)
\end{array}\right\}. 
\]

First remark that
$C_t \subset \{ |z_t -\bz_t | \leq C(\Winf(t) + \tau)
\}$, so that
its volume is bounded by $C(\Winf(t) + \tau)^{2d}$.
From the previous steps, it only remains to bound 
\[
 \int_{z_t \in C_t} I_{bc}(t,\bz_t,z_t) \,dz_t.
\]

We begin by the cut-off case, which is the simpler one.  In
that case, one simply bound $I_{bc}\leq C\,\ep^{ -\bar m \alpha}$ which
implies
\begin{equation}
\int_{z_t \in C_t} I_{bc}(t,\bz_t,z_t) \,dz_t \leq 
C(\Winf(t) + \tau)^{2d} \ep^{ -\bar m \alpha}. \label{Ctcutoff}
\end{equation}

\medskip
It remains the case without cut-off. We denote
$ \tilde C_t=\{j\,|\; \exists z_t\in C_t,\ s.t.\ Z_j(t)=T^t(z_t)\}$, and
transform the integral on $C_t$ in a discrete sum 
\[
\int_{z_t \in C_t} I_{bc}(t,\bz_t,z_t) \,dz_t =  
\sum_{j \in \tilde C_t} a_{ij} I_{Nc}(t,i,j)  \quad \text{ with } 
I_{Nc}(t,i,j) = \frac1\tau \inttau
\frac {dz_t}{|X_i(s) -
X_j(s)|^{\alpha}} \,ds,
\]
where $i$ is the number of the particle associated to $\bz_t$
($T^t(\bz_t)=Z_i(t)$) and
\[
a_{ij} = |\{ z_t \in C_t , \;  T^t(z_t) =
Z_j(t) \}|, \quad \text{so that } \sum_{j \in \tilde C_t} a_{ij} =
|C_t|.
\]
To bound $I_{Nc}$ over $\tilde C_t$, we do another decomposition in $j$. 
Define
\begin{eqnarray*}
JX_t &=& \left\{j\in \tilde C_t\,, \; |X_j(t)-X_i(t)|\geq \frac{d_N(t)}2
\right\}, \\
JV_t &=& \left\{j\in\tilde C_t\,,\;|X_j(t)-X_i(t)|\leq  
|V_j(t)-V_i(t)| \text{ and } |V_j(t)-V_i(t)| \geq \frac{d_N(t)}2 \right\}.
\end{eqnarray*}
By the definition of the minimal distance in $\R^{2d}$, $d_N(t)$, one
has that $\tilde C_t= JX_t\cup JV_t$. 
Since
\[
|T^t(z_t)-z_t|\leq \Winf(t),
\]
one has by the definition of $\tilde C_t$ and $C_t$ that 
for all $j \in \tilde C_t, \ |Z_j(t)-Z_i(t)|\leq C\,(\Winf(t)+\tau)$.

\medskip

Let us start with the bound over $JX_t$. If $j\in JX_t$, one has that
\[
|X_j(s)-X_i(s)|\geq |X_j(t)-X_i(t)|-\int_s^t |V_j(u)-V_i(u)|\,du.
\]
On the other hand, for $u \in [s,t]$, 
\[
|V_j(u)-V_i(u)|\leq 2\Winf(t) + |\bv_u-v_u| \leq 2 (\Winf(t) + \tau
|E|_\infty)+ |\bv_t- v_t| \leq C(\Winf(t)+\tau).
\]
Therefore assuming that with that constant $C$
\begin{equation}
C\,\tau (\Winf(t)+\tau)\leq d_N(t)/4,\label{assJX}
\end{equation}
we have that for any $s\in [t-\tau,\ t]$, $|X_j(s)-X_i(s)|\geq
d_N(t)/4$. Consequently for any $j\in JX_t$
\begin{equation}
I_{Nc}(t,i,j)\leq C\, [d_N(t)]^{-\alpha}.\label{boundJX}
\end{equation}

\medskip
For $j\in JV_t$, we write
\[
|(V_j(s)-V_i(s)) - (V_j(t)-V_i(t))| \leq  \int_s^t
|E_N(X_j(u))-E_N(X_i(u))|\,du.
\]
Note that
\begin{equation}\begin{split}
 |X_j(s)-X_i(s)|&\leq |X_j(t)-X_i(t)|+\int_{s}^t
 |V_j(u)-V_i(u)|\,du\\
&\leq C(W_\infty(t)+\tau)+2\int_{s}^t (W_\infty(u)+R(u))\,du \\
&\leq C(W_\infty(t)+\tau).
\end{split}\label{xclose}
\end{equation}
Hence we get for $s\in [t-\tau,t]$
\[
\int_s^t |E_N(X_j(u))-E_N(X_i(u))|\,du\leq C\,\tau\, |\nabla^N E|_\infty \,
(\Winf(t)+\tau+\ep^{1+r'}).
\]
Note that the constant $C$ still does not depend on $\tau = \ep^{r'}$.
Therefore provided that with the previous constant $C$
\begin{equation}
2 C\,\tau\,|\nabla^N E|_\infty\, (\Winf(t)+\tau)
\leq d_N(t)/4,\label{assJV}
\end{equation}
one has that
\[
|V_j(s)-V_i(s) - (V_j(t)-V_i(t))|\leq d_N(t)/4 \quad  \text{and also } \quad 
|V_i(s) - V_j(s) | \geq \frac{d_N(t)}4.
\]
As in the step for $B_t$  (See equation~\eqref{eq:dispB}) this implies 
the dispersion estimate \\
\mbox{$|X_j(s)-X_i(s)|\geq |s-s_0| \,d_N(t)/4$} for
some $s_0\in [t-\tau,\ t]$. As a consequence for $j\in JV_t$,
\begin{equation}
I_{Nc}(t,i,j)\leq \frac{C}{\tau}\,
  (d_N(t))^{-\alpha}\,\inttau
\frac{ds}{|s-s_0|^\alpha}\leq C\,\tau^{-\alpha}\,(d_N(t))^{-\alpha}.
\label{boundJV}
\end{equation} 

Summing \eqref{boundJX} and \eqref{boundJV}, one gets
\[
\sum_{j \in \tilde C_t} a_{ij} I_{Nc}(t,i,j) \leq
C\,|C_t|\;\left((d_N(t))^{-\alpha} +
\tau^{-\alpha}\, \,(d_N(t))^{-\alpha}\right).
\]
Coming back to $I_{bc}$, using the bound on the volume of $|C_t|$ and keeping
only the largest term of the sum
\begin{equation}
\int_{C_t} I_{bc}(t,\bz_t,z_t)\,dz_t\leq
C\,(\Winf(t)+\tau)^{2d}\tau^{-\alpha}\, \,(d_N(t))^{-\alpha}
.\label{Ct}
\end{equation}

%
%
%
\subsubsection{Conclusion of the proof of Lemmas \ref{boundIK},
  \ref{boundIcut}}
%
Assumptions \eqref{assJX} and \eqref{assJV} are ensured by the
assumptions of the lemma. Summing up \eqref{At} for $\Ialp$ or \eqref{At2}
for $\Jalp$, with \eqref{Ic}, \eqref{Bt}
and \eqref{Ct}, we indeed find the conclusion of the first lemma.

In the $S^\alpha_m$ case, no assumption is needed,
 and summing up the
bounds \eqref{At}, \eqref{Ic}, \eqref{Btcutoff}, \eqref{Ctcutoff}, 
we obtain the second lemma.

\subsection{A bound on $\Winf(\mu_N,f_N)$  in the case without
cut-off}
%
%
In this subsection, in order to make the argument clearer, we number
explicitly the constants. 
Let us summarize the important information of Prop. \ref{propeasy} and Lemma
\ref{boundIK}.
Let us also rescale the interested quantities s.t. all may be of order $1$
\[
\ep\,\tilde W_\infty(t)= \Winf(t),\quad  \ep^{1+r}\,
\tilde d_N(t)=d_N(t).
\]
Remark that by Proposition~\ref{prop:Winfbound} $\tilde W_\infty(t) =  c_\phi
>0$.
By assumption $(i)$ in Theorem \ref{thm:deter}, also note that $\tilde
d_N(0)\geq 1$.

Recalling $\tau=\ep^{r'}$ (with $r' >r>1$), the condition of
Lemma~\ref{lem:boundI}
after rescaling reads
\begin{equation}
C_1\, \ep^{r'-r}\,
(1+|\nabla^N E|_\infty(t))\,\tilde W_\infty(t)\leq \,\tilde d_N(t).
\label{assumption}
\end{equation}
In Lemma~\ref{lem:boundI}, we proved that there exist some constants $C_0$ and
$C_2$ independent of $N$ (and hence $\ep$), such that if \eqref{assumption} is
satisfied, then for any $t\in [0,\ T]$
\begin{eqnarray*}
& \tilde W_\infty(t) &\leq \tilde W_\infty(t-\tau)+C_0\,\ep^{r'}\, \left(\tilde
  W_\infty(t)+\ep^{\lambda_1}\,\tilde W_\infty^d(t) +  \ep^{\lambda_2}
\,\tilde W_\infty^{2d}(t)\, \tilde d_N^{-\alpha}(t) \right),\\
&|\nabla^N E|_\infty(t) &\leq   C_2 \left(1+ 
\ep^{\lambda_3}\,\tilde W_\infty^d(t)
+  \ep^{\lambda_4}
\,\tilde W_\infty^{2d}(t)\,\tilde d_N^{-\alpha}(t)) \right)\\
&  \tilde d_{N}(t)  + \ep^{r'-r} & \geq [  \tilde d_{N}(t- \tau) + \ep^{r'-r}]
e^{-\tau (1+ \dENinf(t))},
\end{eqnarray*}
where $\ep$ appear four times with four different exponents $\lambda_i, i
=1,\ldots,4$ defined by
\begin{align*}
 &\lambda_1= d-1-\alpha\,r', &\lambda_2 = 2d-1-\,\alpha(1+r'+r),\\
& \lambda_3= d- 1 - r' -\alpha\,r', &\lambda_4 = 2d- 1 - r'-
\,\alpha(1+r'+r). 
\end{align*}
To propagate uniform bounds as $\ep\rightarrow 0$ and
$N\rightarrow\infty$,  we
need all $\lambda_i$ to be positive. As $r,r'>0$, it is
clear that $\lambda_1 > \lambda_3$ and 
$\lambda_2 > \lambda_4$. Thus we need only check $\lambda_3>0$ and
$\lambda_4>0$. 
As $r'>r$, it is sufficient to have
\[
r' < \frac{d - 1}{1+ \alpha}, 
\quad \text{and}\quad  r' < \frac{2d-1-\alpha}{1+
2\,\alpha}.
\]
Note that a simple calculation shows that
\[
\frac{d - 1}{1+ \alpha} -\frac{2d-1-\alpha}{1+2\,\alpha} = \frac{\alpha^2 -
d}{(1+\alpha)(1+ 2\alpha)} < 0,
\]
so that the first inequality is the stronger one. 
Thanks to the condition given in Theorem \ref{thm:deter},
$r < r^\ast := \frac{d-1}{1+\alpha}$, so that if we choose any $r' \in (r ,
r^\ast)$,  the corresponding $\lambda_i$ are all positive. We fix a $r'$
as above and denote $\lambda= \min_i(\lambda_i)$. \Black
Then by a rough estimate
\begin{equation}
\begin{split}
 \tilde W_\infty(t) &\leq \tilde W_\infty(t-\tau)+C_0\,\tau \left(\tilde
  W_\infty(t)+2  \,\ep^{\lambda}\,\tilde
W_\infty^{2d}(t)\,d_N^{-\alpha}(t)\right),\\
|\nabla^N E|_\infty(t) &\leq   C_2\left( 1 \,+2 
\,\ep^{\lambda}\,\tilde
W_\infty^{2d}(t)\,\tilde d_N^{-\alpha}(t) \right),\\
 \,\tilde d_{N}(t) & \geq [ \,\tilde d_{N}(t-\tau) + \ep^{r'-r}]
e^{-(1+ \dENinf(t))\tau} -\ep^{r'-r}.
\end{split}\label{roughestimate}
\end{equation}
If for some $t_0 >0$ one has \eqref{assumption} on the whole time interval 
$[0,t_0]$ and
\begin{equation}
\forall t \in [0,t_0],\quad
2\,\ep^{\lambda}\,\tilde
W_\infty^{2d}(t)\,\tilde d_N^{-\alpha}(t)\leq 1,\label{assumption2}
\end{equation}
then we get  $\tilde W_\infty(t)\leq \tilde
W_\infty(t-\tau)+ 2 C_0 \tau \tilde W_\infty(t)$ so that if $2 C_0 \tau <1$
\begin{equation}
\begin{split}
 \tilde W_\infty(t) &\leq \tilde W_\infty(t-\tau)(1 - 2 C_0 \tau)^{-1},\\
|\nabla^ N E|_\infty(t) &\leq 2\,C_2,\\
\tilde d_N(t)&\geq \,e^{-(1+2C_2)\,t} - \ep^{r'-r}
\end{split}\label{presquegronwall}
\end{equation}
for any $t \in [0,t_0]$. The last inequality implies $\tilde d_N(t)\geq \frac12
\,e^{-(1+2C_2)\,t}$ if $2\ep^{r'-r} e^{(1+2C_2)\,T} <1$. That condition is
fulfilled for $\ep$ small enough, i.e. $N$ large enough : $\ln N \geq C T$. 

The first inequality in \eqref{presquegronwall}, 
iterated gives $\Winf(t) \leq \Winf(0) (1 - 2 C_0 \tau)^{-\frac
t \tau}$. If $C_0 \tau \leq \frac14$, then we can use $-\ln(1-x) \leq 2x$ for
$x \in [0, \frac12]$, and get
$$ \tilde \Winf(t) \leq \tilde W_\infty(0)e^{ 4C_0 t} $$ 
To summarize, under the previous assumption it comes for all $t \in [0,t^0]$
\begin{equation}
\begin{split}
 \tilde W_\infty(t) &\leq e^{ 4C_0 t},\\
|\nabla^ N E|_\infty(t) &\leq 2\,C_2,\\
\tilde d_N(t)&\geq \frac12 \,e^{-(1+2C_2)\,t}.
\end{split}\label{result} 
\end{equation}
As we mention in the introduction, we only deals with
continuous solutions to the N particles system~\eqref{eq:ODE}. So $\tilde W_\infty(t)$ and
$\tilde d_N(t)$ are continuous functions of the time, and $\dENinf(t)$ is also continuous in time thanks to the smoothing parameter that appears in its definition~\eqref{def:dEinfty}.

As we explain in Remark~\ref{rem:time0}, $\dENinf(0)=0$ and the 
conditions~\eqref{assumption} and~\eqref{assumption2} are
satisfied at time $t=0$. In fact, at time $0$ they maybe rewritten
$$
C_1 \ep^{r'-r} \tilde W_\infty(0) \le \tilde d_N(0),
\qquad
2 \ep^\lambda \tilde W_\infty(0)^{2d} \tilde d_N(0)^{-\alpha} \le 1
$$
and this is true for $N$ large enough because of our assumption on $\ep$ and
$d_N(0)$. Then by continuity there exists a maximal time $t_0\in ]0,\ T]$
(possibly $t_0=T$) such that they are satisfied on $[0,\ t_0]$.

We show that for $N$ large enough, {\it i.e.} $\ep$ small enough, then
one necessarily has $t_0=T$. Then we will have \eqref{result} on
$[0,\ T]$ which is the desired result.
This is simple enough. By contradiction if $t_0<T$ then
\[
C_1\, \ep^{(r-r')}\,
(1+|\nabla^N E|_\infty(t_0))\,\tilde W_\infty(t_0)= \tilde d_N(t_0),\quad
\text{or} \quad \ 4\,
\,\ep^{\lambda}\,\tilde
W_\infty^{2d}(t_0)\,\tilde d_N^{-\alpha}(t_0)= 1.
\]
But until $t_0$, \eqref{result} holds. Therefore
\[
\ep^{\lambda}\,\tilde
W_\infty^{2d}(t_0)\,\tilde d_N^{-\alpha}(t_0) \leq \ep^{\lambda}\,
2^\alpha\,e^{(\alpha\,+\,(4d+2\alpha)\max(C_0,C_2))\,t_0}<1,
\] 
for $\ep$ small enough with respect to $T$ and the $C_i$. This is the
same for \eqref{assumption}
\[
C_1\, \ep^{(r-r')}\,
(1+|\nabla^N E|_\infty(t_0))\,\tilde W_\infty(t_0) \tilde d_N^{-1}(t_0) \leq
2 \ep^{(r-r')} C_1(1+2C_2) e^{(1+6\max(C_0,C_2)) t_0} <1.
\]
Hence we obtain a contradiction and prove that
\begin{equation} \label{concl:Winf}
\forall t \le T, \quad \Winf(f_N(t),\mu_N(t)) \le e^{4 C_0 t } \Winf(f_N^0,\mu_N^0),
\end{equation}
for $N$ large enough.

\subsection{A bound on $\Winf(\mu_N,f_N)$  in the case with
cut-off}
In the cut-off case, using Lemma
\ref{boundIcut} together with the inequality $ii)$ of the Proposition
\ref{propeasy}, we may obtain
\[
 \Winf(t) \leq \Winf(t-\tau) + C_0 \Winf(t) \left[1+
(\Winf(t)+ \tau )^{d-1}\tau^{-1} \ep^{\bar m(1-\alpha)} + (\Winf(t)+ \tau
)^{2d-1} \ep^{- \bar m \alpha} \right] .
\]
We again rescale the quantity
$ \Winf(t) = \ep \tilde \Winf(t) $.
Choosing in that case  $\tau=\ep$, it comes for
$1\leq\alpha<d-1$,
\[
\tWinf(t) \leq \tWinf(t-\tau) + C_0 \tWinf(t) \tau \left[1+
\ep^{d-2 - \bar m(\alpha - 1)} \, \tWinf^{d-1}(t) + \ep^{2d - 1 - \bar m \alpha}
\,
\tWinf^{2d-1}(t)  \right] .
\]
As in the previous section, we will get a good bound provided that the
power of 
$\ep$ appearing in parenthesis are positive. The two conditions read
\[
 \bar m <  \bar m^* := \min \left(\frac{d-2}{\alpha -1} \,, \frac{2d-1}{\alpha}
 \right). 
\]
In that case, for $N$ large enough (with respect to $e^{Ct}$), we get a control
of the type
$$
\frac d {dt} \tWinf(t) \leq 4 C_0 \tWinf(t),
$$
(but discrete in time) which gives that
\begin{equation} \label{concl:Winf'}
\forall t \le T, \quad \Winf(f_N(t),\mu_N(t)) \le e^{4 C_0 t } \Winf(f_N^0,\mu_N^0),
\end{equation}
for $N$ large enough.
\Black

\begin{remark}
 In the cut-off case (and also in the case without cut-off), it seems important
to be able to say that the initial configurations $Z$ we choose have a total
energy close from the one of $f^0$. Because, if the empirical distribution
$\mu_N^\zz$ is close from $f^0$, but has a different total energy, we would not
expect that they remain close a very long time.
Fortunately, such a result is true and under the assumptions of
Theorem~\ref{thm:deter}~and~\ref{thm:cutoff}, the total energy of the empirical
distributions is close from the total energy of $f^0$. 

Unfortunately, the proof is not simple. But, it can be done
using the argument presented here for the deterministic theorems. First, the
difference
between the kinetic energies is easily controlled because our solutions are
compactly supported and that there is no singularity there. 
Next, performing calculations very similar to the ones done in the proofs, we
can
control the difference between a small average in time of the potential
energies, 
on the small interval of time $[0, \tau]$. Then, we control the average of the
total energy, which is constant.  

\end{remark}

\subsection{Estimation of the distance $W_1(f,\mu_N)$.}

\paragraph{The case without cut-off.}
Just apply the stability estimate for solutions of Vlasov equation given by
Proposition~\ref{Loeper}. This is possible since the uniform bound on $\|f_N\|_\infty$ 
given by point $ii)$ in Theorem~\ref{thm:deter} and~\ref{thm:cutoff}, and the uniform
 bound on the size of the support of Proposition~\ref{propeasy}, implies an uniform bound on $\| \rho_N \|_\infty$. We get
\begin{align*}
W_1(f,f_N) & \le e^{C_0 t} \, W_1(f^0,f^0_N)\, \\
 & \le 
e^{C_0 t} \, \bigl(W_1(f^0,\mu^0_N) + W_1(\mu_N^0,f^0_N) \bigr), \\
& \le
e^{C_0 t} \, \Bigl(W_1(f^0,\mu^0_N) + N^{-\frac\gamma{2d}} \Bigr)
\end{align*}
This together with the bound~\eqref{concl:Winf} concludes the proof since
\begin{align*}
W_1(f,\mu_N) &\le W_1(f, f_N) + W_1(f_N, \mu_N) \\
& \le W_1(f, f_N) + W_\infty(f_N, \mu_N) \\
& \le e^{4 C_0 t} \Bigl(W_1(f^0,\mu^0_N) + 2\,  N^{-\frac\gamma{2d}} \Bigr).
\end{align*}

\paragraph{The case with cut-off.}
Proposition~\ref{prop:WPvlasov} implies that the strong solution  $f$ with initial data $f^0$ is also defined at least on $[0,T^\ast)$. 
And from the condition $(S^\alpha_m)$ restated in~\eqref{eq:Ckappa'} in term
of $\ep$, we get that
$$
\| F - F_N \|_1 \le \ep^{\bar m (d- \alpha) } \le \ep,
$$
since $\bar m \ge 1$ and $d- \alpha \ge1$.
So we can apply the stability estimate given by
Proposition~\ref{Loeper} with $F_1 =F$ and $F_2=F_N$ and get that 
\begin{align*}
W_1(f,f_N)  & \leq  \,e^{C_0 t}  \bigl( W_1(f^0,f^0_N)+ \ep \bigr) \\
& \leq 
e^{C_0 t}  \Bigl(W_1(f^0,\mu^0_N) + W_1(\mu_N^0,f_N^0) +  N^{-\frac\gamma{2d}} \Bigr), \\
& \leq 
e^{C_0 t}  \Bigl(W_1(f^0,\mu^0_N) + 2  \, N^{-\frac\gamma{2d}} \Bigr).
\end{align*}
With the bound~\eqref{concl:Winf'} it leads to
$$
W_1(f,\mu_N)   \leq  \,e^{4 C_0 t} \Bigl(W_1(f^0,\mu^0_N) + 3  \, N^{-\frac\gamma{2d}} \Bigr),
$$
and this conclude the proof in the cut-off case.
\Black

\section{From  deterministic results (Theorem \ref{thm:deter} and
\ref{thm:cutoff}) to propagation of chaos.}

The assumptions made in Theorem~\ref{thm:deter} are in some sense generic, when the initial positions and speeds are
chosen with the law $(f^0)^{\otimes N}$. Therefore, to prove Theorem
\ref{thm:prob} from Theorem \ref{thm:deter}, we need to 
\begin{itemize}
 \item Find  a good choice of the parameters $\gamma$ and $r$ so that there is a small probability that empirical measures, chosen with the law
$(f^0)^{\otimes N}$, do not satisfy the conditions $i)$ and $ii)$ of Theorem \ref{thm:deter}, and are far away from $f^0$ in $W_1$
distance;
 \item Apply Theorem~\ref{thm:deter} on the complementary set that is almost of full measure.
\end{itemize}
For the first point, we will use results detailed in the next two
sections.  

\subsection{Estimates in probability on the initial distribution.}

\paragraph{Deviations on the infinite norm of the smoothed empirical
distribution $f_N$.} 
The precise result we need is given by the Proposition~\ref{prop:largedev} in
the Appendix.
It tells us that if the approximating kernel is $\phi = {\mathbf 1}_{[-\frac12,
\frac12]^{2d}}$, then 
 $$
 \PP \left(\|f_N^0\|_\infty \geq 2^{1+2d} \| f^0\|_\infty \right)
 \leq C_2 N^{\gamma}  e^{ - C_1 N^{1 - \gamma}}.
$$
with $C_2 =(2R^0+ 2)^{2d}$, $R^0$ the size of the support of $f$, and 
$C_1=\left( 2 \ln 2 - 1  \right)  2^{2d} \| f\|_\infty$.

\medskip

We would like to mention that we were first aware of the possibility of getting
such estimates in a paper of  Bolley, Guillin and Villani \cite{BolGuiVil07},
where the authors obtain 
quantitative concentration inequality for $\| f^N - f \|_\infty$ under the additional assumption that $f^0$ and $\phi$ are Lipschitz.
Unfortunately, they
cannot be used in our setting because they would require too large a smoothing
parameter. Gao obtains in \cite{Gao03}  large (and moderate) deviation principles for $\|
f^N -f \|_\infty$. But a large deviation principle is too precise for our purpose, and also less convenient since it provides only an asymptotic estimate, and no quantitative bounds.
Finally, we choose to prove a more simple estimate that is well adapted to our problem.

\paragraph{Deviations for the minimal inter-particle distance.}
It may be proved
with simple arguments that the scale $\eta_m$ is almost surely larger than
$N^{-1/d}$ when $f^0 \in L^\infty$. A precise result is stated in the
Proposition below, proved in \cite{Hau09}
\begin{prop} \label{prop:dN}
There exists a constant $c_{2d}$ depending only on the dimension such
that if $f^0 \in L^\infty(\R^{2d})$, then  
\[
\mathbb{P}\left( d_N(Z)  \geq \frac l {N^{1/d}} \right) \geq
e^{-c_{2d} \|f^0\|_\infty l^d} .
\] 
\end{prop}
We point out that  this is not a large deviation result : the inequalities are in
the wrong direction.  This is quite natural because $d_N$ is not a average quantity, but an infimum.
It is that condition that prevents us from obtaining a
``large deviation'' type result in Theorem \ref{thm:prob}, contrarily to the
cut-off case of Theorem \ref{thm:probcutoff}.  In fact, the only bound it
provides on the ``bad'' set is
$$
\mathbb{P}\left( d_N(Z)  \leq \frac l {N^{1/d}} \right) \leq
1 - e^{-c_{2d} \|f^0\|_\infty l^d}  \leq c_{2d} \|f^0\|_\infty l^d.
$$ 
With the notation of Theorem \ref{thm:deter} it comes that if $s= \gamma
\frac{1+r}2 -1 >0$ then 
\begin{equation} \label{dN}
\mathbb{P}\left( d_N(Z)  \leq \ep^{1+r} \right)  = 
\mathbb{P}\left( d_N(Z)  \leq \frac {N^{-s/d}} {N^{1/d}} \right) \leq 
c_{2d} \|f^0\|_\infty  N^{-s}.
\end{equation}
\medskip

\paragraph{Deviations for the $W_1$ MKW distance.} 
It is more or less classical that if the $Z_i$ are independent random variables with identical law $f$, the empirical measure $\mu^Z_N$  goes in probability to $f$. This theorem is known as the empirical law of large number or Glivenko-Cantelli theorem and is due in this form to Varadarajan \cite{Varada}. But, the convergence may be quantified in Wasserstein distance, and recently upper bound on the large deviations of $W_1(\mu_N^Z,f)$ were obtained by Bolley, Guillin and Villani \cite{BolGuiVil07} and 
Boissard \cite{BoissardPhD}.  However the first one concerns only very large deviations, and the last result is more interesting for our purpose. 
\begin{prop}[Boissard \cite{BoissardPhD}, Annexe A, Proposition 1.2 ]\label{probaint}
Assume that $f$ is a non negative measure compactly supported on $B_{2d}(0,R) \subset \R^{2d}$. If $d
\geq 2$, and the $Z=(Z_1,\ldots,Z_N)$ are chosen according to the
law $(f^0)^{\otimes N}$, then there is an explicit constant $C_1 = 2^{-(2d+1)} R^{-2d}$, such that the associated empirical measures $\mu_N^Z$ satisfy
\[
\PP \Bigl( W_1(\mu_N^Z,f)  \geq  \E[ W_1(\mu_N^Z,f) ] + L  \Bigr) \leq
e^{-C_1  N L^2}.
\]
Since it is already known (see \cite{BoissardPhD} or \cite{Dobric} and references therein) that for $d \ge 2$ there exists a  numerical constant $C_2(d)$ such that 
\[
\E [ W_1(\mu_N^Z,f) ] \le C_2 \frac R{N^{1/2d}}, 
\]
the previous result with $L = C_2 \frac R{N^{1/(2d)}}$ implies that for $C_3(R,d) := C_1(R) C_2(d)^2 R^2$, 
\begin{equation} \label{eq:W1largedev}
\PP \left( W_1(\mu_N^Z,f)  \geq  2 \frac {C_2 R} {N^{1/2d}}  \right) \leq
e^{-C_3 N^{1 - 1/d}}.
\end{equation}
\end{prop}
\Black
\subsection{From Theorem \ref{thm:deter} to Theorem \ref{thm:prob}}
Now take the assumptions of Theorem \ref{thm:prob} : $F$ satisfies a $(S^\alpha)$ condition for 
$\alpha<1$ and $f^0\in L^\infty$ with support in  some ball $B_{2d}(0,R_0)$ in dimension $d \ge 3$. We choose 
$$
\gamma \in \left( \gamma^\ast = \frac{2 + 2\alpha}{d+\alpha} ,1\right), \text{
  and} \quad r \in \left( \frac2 \gamma -1 , r^\ast=\frac{d-1}{1 + \alpha}
\right),
$$ 
the condition on $\gamma$ ensuring that the second interval is non
empty. We also define  
$$
s :=\gamma\,\frac{1+r}{2}-1 >0,  \quad
\lambda= 1- \max \left(\gamma,\frac1d \right).
$$ 
Denote by $\omega_1$, $\omega_2$ the sets of
initial conditions s.t.  respectively 
$(i)$ and $(ii)$ 
of Theorem
\ref{thm:deter} hold
and $\omega_3$ 
s.t. $W_1(\mu_N,f^0)  \leq \frac{1}{N^{\gamma/(2d)}}$. Precisely 
\[ \begin{split}
\omega_1 := \{ Z^0  \text{ s.t. } d_N(Z^0) & \geq \ep^{1+r}  \}, \quad \omega_2 :=
\{ Z^0  \text{ s.t. } \|f_N^0\|_\infty \leq 2^{1+2d} \|f^0\|_\infty  \} \\
& \omega_3 := \{ Z^0  \text{ s.t. } W_1(\mu_N^0 ,f^0) \leq \ep \}
\end{split} \]
 By the results stated in the previous section, one knows that for $N \ge (2C_2 R)^{2d/(1-\gamma)}$
\begin{equation} \label{setomega}
\PP(\omega_1^c)\leq C\, N^{-s},\quad 
\PP (\omega_2^c) \leq C N^\gamma e^{ - C N^{1 - \gamma}}, \quad 
 \PP(\omega_3^c) \leq \;  e^{- C N^{1- \frac1d}}.
 \end{equation}
Denote $\omega=\omega_1\cap\omega_2\cap\omega_3$. Hence $|\omega^c|\leq
|\omega_1^c|+|\omega_2^c|+|\omega_3^c|$ and for $N$ large enough
\begin{equation} \label{boundomega}
 \PP(\omega^c)\leq C\, N^{-s} + C\, N^\gamma \, e^{- C\,N^{1-\gamma}} + e^{- C N^{1- \frac1d}}
\leq C\,
N^{-s},
\end{equation}
and checking carefully the dependence, we can see that the  constant $C$ depends only on $d,R,\|f^0\|_\infty,\gamma$.
Since we known that global solutions to the $N$ particles system~\eqref{eq:ODE} exist for almost all initial conditions (see the discussion on this point in subsection~\ref{subsec:exis}),  one may apply Theorem \ref{thm:deter} to $(f^0)^{\otimes N}$-a.e. initial condition in $\omega$ and get on $[0,\ T]$
\[
W_1(f,\mu_N)\leq
e^{C_0t} \Bigl( 2 \,W_1(f,\mu_N^0) + N^{-\frac\gamma{2d}} \Bigr)
\leq 3 \, e^{C_0t} \, N^{-\frac\gamma{2d}},
\]
which proves that 
\[
\omega \subset \Bigl\{ \forall t \in  [0, T], \; W_1(f,f_N)\leq
\frac{3e^{C_0t}}{N^{\gamma/(2d)}}\Bigr\}.
\]
The bound \ref{boundomega} then gives Theorem \ref{thm:prob}. 
%
\subsection{From Theorem
\ref{thm:cutoff} to Theorem \ref{thm:probcutoff}}
%
In the cut-off case, one can derive Theorem \ref{thm:probcutoff} from Theorem
\ref{thm:cutoff} in the same manner. As we do not use the minimal distance in
that case, the proof is simpler and we get a stronger
convergence result. 

We only have to consider $\omega = \omega_2 \cap \omega_3$, where $\omega_2$ and $\omega_3$ are defined according to~\eqref{setomega}. Then, the bound~\eqref{boundomega} is replaced for $N$ larger than an explicit constant by 
\begin{equation} \label{boundomegabis}
 \PP(\omega^c) \leq C\, N^\gamma \, e^{- C\,N^{1-\gamma}} + e^{- C N^{1- \frac1d}}
 \leq  C\, N^\gamma \, e^{- C \,N^{-\lambda}}
\end{equation}
Next, for any $Z^0 \in \omega$, we can apply Theorem~\ref{thm:cutoff} 
and obtain the stability estimate for any $T < T^\ast$
\[ 
W_1(f,\mu_N)\leq 2\,  e^{C_0t} \Bigl( W_1(f^0,\mu_N^0) + N^{-\frac\gamma{2d}} \Bigr)
\le 4\,  e^{C_0t} \, N^{-\frac\gamma{2d}}.
\]
From there, we obtain as before that for $N$ large enough
$$ \PP \left( \exists \, t \in [0,T], \; W_1(\mu_N(t),f(t)) \geq
\frac{4e^{C_0t}}{N^{\gamma/(2d)}} \right) \leq   C N^\gamma e^{-C N^\lambda}.
$$
Replacing $2^{1+2d}$ by any $\lambda>1$ in the definition of $\omega_2$, we may also get estimates that are valid till a time $T^*$ as large as possible.

%
%
%

\section{Related Discussions}

%
\subsection{The question of existence of solutions to System~\eqref{eq:ODE}.} 
\label{subsec:exis}
We have just mentioned till now the most basic question for System~\eqref{eq:ODE} with a
singular force kernel,
namely  whether one can even expect to have solutions to the system for a fixed
number of particles. 

Since we only use forces that are singular only at the origin, the  usual 
Cauchy-Lipschitz theory implies that starting from any initial
conditions such that $X^N_i \neq X^N_j$ for all $i \neq j$, there exists a unique local
solution, defined till the time of first collision time $T^\ast$, when for some
couple $i,j$ we have $X_i^N(T^\ast)=X^N_j(T^\ast)$.  Unfortunately this time $T^\ast$ depends on the initial configuration (and thus on $N$) and could be very small.

In the case where the 
interaction force $F$ derives from a repulsive
singular potential $\phi$ {\em strong enough}, i.e. if $\Phi$ satisfy $\lim_{x \rightarrow 0}
\Phi(x) = + \infty$, then collisions can never occur and the solutions given
by the classical Cauchy-Lipschitz theory are global, i.e. $T^\ast = +\infty$ for all initial configurations.

In the other cases, it is not possible to extend the local result in such a simple way. One could try to apply the DiPerna Lions theory
\cite{DipLions}, that allows to handle vector fields that are locally in
$W^{1,1}$. This looks promising since any force satisfying the
condition~\eqref{eq:Calpha}, with $\alpha < d-1$ has the required local
regularity. But unfortunately, the DiPerna-Lions theory also requires a
condition on the growth of the vector-field at infinity, which is not satisfied in our case.
However if  the interaction forces $F$ derives from a potential $\Phi$ which is bounded at the origin (without any sign condition), the DiPerna Lions theory still leads to global solutions for almost every
initial conditions.  This is stated precisely in
the following Proposition which is a consequence of \cite[Theorem 4]{Hau04}.
\begin{prop} \label{prop:uniqODE}
Assume that $F = -\nabla \Phi$ with  $\Phi \in W^{2,1}_{loc}$, and that 
$ \Phi(x) \ge - C(1+ |x|^2)$ for some constant $C>0$.
Then for any fixed $N$, there exists a unique measure preserving and energy
preserving flow defined almost everywhere on $\R^{2dN}$ associated
to~\eqref{eq:ODE}. Such a flow precisely satisfy
\begin{itemize}
 \item [i)] there exists a set $\Omega\subset \R^{2dN}$ with
$|\Omega|=0$
s.t. for any initial data $Z^0 \in
\R^{2dN}\setminus\Omega$, we have a trajectory
$Z(t)$ solution to~\eqref{eq:ODE},
\item[ii)] for a.e. trajectory the energy conservation is satisfied,
\item[iii)] the family of solutions defines a global flow, which preserves the
measure on $\R^{2dN}$.
\end{itemize}
\end{prop}
Remark that if $F=-\nabla \Phi$ then the conditions on $\Phi$ are fulfilled whenever $F$ satisfies \eqref{eq:Calpha}. So this proposition implies the global existence of solutions for almost all initial positions and velocities in that case, and this is completely sufficient for our results: Theorem~\ref{thm:deter} requires only the existence of a solution with given initial data and Theorem~\ref{thm:prob} requires the existence of solution for almost all initial data. 

\medskip
In the case of some specific but more singular attractive potentials as the gravitational force ($\alpha =d-1$) in dimension $2$ or $3$, and also for some others power law forces, it is known \cite{Saari} that initial conditions leading to ``standard'' collisions (possibly multiple and simultaneous), is of zero measure. But, what is unknown even if it seems rather natural, is that the set of initial collisions leading to the so-called ``non-collisions'' singularities, which do exists \cite{Xia}, is also of zero measure for $N \ge 5$. Up to our knowledge it has only been proved for $N \le 4$ \cite{Saari2}. In fact, there is a large literature about this $N$ body problem in the physicist and mathematician communities. However, that  discussion is not really relevant here since in the ``strongly'' singular case $\alpha \in [1,d-1)$, we use a regularization or cut-off of the force (see the condition~\eqref{eq:Ckappa}), thanks to which the question of global existence becomes trivial.  

\medskip
Eventually, the only case in which we are not covered by the existing literature is the case of non potential force satisfying the $(S^\alpha)$-condition for some $\alpha <1$,
for which we claimed a result without cut-off.  In that case, we opt for the following simple strategy.
As in the case with larger singularity we use a cut-off or regularization of the interaction force. The existence of global solution is then straightforward. And our results of convergence are valid independently of the size of cut-off (or smoothing parameter) which is used. It can be any positive function of the number of particles $N$.

Note that this  suggests in fact that for a fixed $N$, the analysis done in this article should imply the existence of solutions for almost all initial conditions. If one checks precisely, ours proofs show that trajectories may be extended after a collision where the relative velocities between the two particles goes to a non zero limit. Hence the only collisions that remain problematic are those where the relative velocity of the colliding particles vanishes, but our result controls the probability of this happening.
This was mentioned in remark~\ref{rmk:exist} after Theorem~\ref{thm:prob}.

\subsection{The structure of the force term: Potential, repulsion, attraction?}
\label{subsec:sign}
In the particular case where the force derives from a potential $F=-\nabla \Phi$, the system \eqref{eq:ODE} is endowed with some important additional structure, for example the conservation of energy
\[
\frac{1}{N}\sum_i \frac{|V_i|^2}{2}+\frac{1}{2N^2}\sum_{i\neq j} \Phi(X_i-X_j)=const.
\] 
When the forces are repulsive, {\em i.e.} $\Phi\geq 0$, this immediately bounds the kinetic energy and separately the potential energy. However this precise structure is never used in this article, which may seem weird at first glance.   We present here some arguments that can explain this fact.

First, for the interactions considered in the case without cut-off, again satisfying a $(S^\alpha)$ condition with $\alpha<1$,  the potential $\Phi$ is  continuous (hence locally bounded). 
In that case the singularity in the force term is too weak to really see or use a difference between repulsive and attractive interactions. Two particles having a close encounter cannot have a strong influence onto each other, both in the attractive or repulsive case.
Similarly the fact that the interaction derives from a potential is not really useful, hence our choice of the slightly more general setting.

It should here be noted that the previous discussion applies to every previous result on the mean field limit or propagation of chaos in the kinetic case: They all require assumptions (typically $\nabla^2\Phi$ locally bounded) implying that the attractive or repulsive nature of the interaction does not matter; the situation is different for the macroscopic ``Euler-like'' cases, see the comments in the paragraph devoted to that case.
The present contribution shows that mean field limits and propagation of chaos are  essentially valid at least as long as the potential is bounded (instead of at least $W^{2,\infty}_{loc}$ as before). This corresponds to the physical intuition that nothing should go wrong as long as the local interaction between two very close particles is too weak to impact the dynamics. 
 
 \medskip
The exact structure of the interaction kernel should become crucial once this threshold is passed, {\em i.e.} for $ \Phi(x)  \sim C  |x|^{1-\alpha}$ at the origin with $\alpha\geq 1$.
But here we use in that case a cut-off, which weaken the effect of the interaction between two very close particles. In fact in order to prove the mean-field limit, we are able to show that if the cut-off is large enough, these local interactions may be neglected. So our techniques still do not make any difference between the repulsive or attractive cases. 

However in the case where the ``strong'' singularity is repulsive, the potential energy is bounded, and if we were able to use this fact, we would obtain results depending of the attractive-repulsive character of the interaction. 
In that respect, we point out that the information contained in a bounded  potential energy is actually quite weak and clearly insufficient, at least with our techniques.  Assume for instance that  $\Phi(x)\sim |x|^{1-\alpha}$ for some $\alpha >1$. Then the boundedness of the potential energy implies that the minimal distance in physical space between any two particles is of order $N^{-2/(\alpha-1)}$, which is at best $N^{-2}$ in the Coulomb case, $\alpha=2$. But it can be checked that the cut-off parameter $N^{-m}$ given in 
Theorem~\ref{thm:probcutoff} as a power $m$ which is always much lower than $\frac2{\alpha-1}$, i.e. that the cut-off we use is always much larger than the minimal distance provided by the bound on the potential energy.  To go further, an interesting idea is to compare the dynamics of the $N$ particles with or without cut-off. But even if the difference between the original force and its mollified version is well localized, it is quite difficult to understand how we can control the difference between the two associated dynamics. We refer to \cite{BarJab} for a first attempt in that direction,  in which well-localized and singular perturbation of the free transport are investigated.

\medskip
Therefore in those singular settings, the repulsive or potential structure of the interaction will only help in a more subtle (and still unidentified) manner. An interesting comparison is the stability in average proved in \cite{BaHaJa}: This requires repulsive interaction not to control locally the trajectories but in order to use the statistical properties of the flow (through the Gibbs equilibrium).

\appendix 
\section{Appendix}

\subsection{Large deviation on the infinite norm of $f_N $.}

\begin{prop} \label{prop:largedev}
Assume that $\rho$ is a probability on $\R^n$ with support included in
$[-R^0,R^0]^n$ and and bounded density $f(x)\,dx$.
Let $\phi$ be a bounded cut-off function, with support in $[-\frac L 2
,\frac L 2]^n$ and total mass one, and define the usual $\phi_\ep :=
\frac1{\ep^n} \phi(\frac \cdot
\ep)$.
For any configuration $Z_N=(Z_i)_{i \leq N}$ we define 
$$
f_N^\zz := \mu_N^\zz \ast \phi_\ep(N).
$$
If $\ep(N) = N^{-\frac \gamma n}$ and the $Z_N$ are distributed according to
$f^{\otimes N}$, then we have the explicit ``large deviations'' bound
with $c_\phi = (2L)^n \| \phi\|_\infty$ and $c_0= (2R^0+ 2)^n L^{-n}$
\begin{equation} \label{eq:largedev}
\forall \beta >1,\quad  \PP \left(\|f_N^\zz\|_\infty \geq \beta c_\phi  \|
f\|_\infty \right)
 \leq c_0 N^{\gamma}  e^{ - \left( \beta \ln \beta -
\beta +1  \right)  (2L)^n \| f\|_\infty N^{1 - \gamma}}.
\end{equation}
In particular, for $\phi = \mathbf 1_{[-1/2,1/2]^n}$ and $\beta =2$, we get
\begin{equation} \label{eq:largedev2}
 \PP \left(\|f_N^\zz\|_\infty \geq 2^{1+n} \| f\|_\infty \right)
 \leq (2R^0+ 2)^n N^{\gamma}  e^{ - \left( 2 \ln 2 - 1  \right)  2^n \|
f\|_\infty N^{1 - \gamma}}.
\end{equation}
\end{prop}

\begin{proof}
For any $Z \in \R^{nN}$ and $z \in \R^n$, we have
\begin{eqnarray*}
 f_N^\zz(z) & = & \frac1N \sum_{i=1}^N \phi_\ep(z - Z_i) = \frac1{N\,\ep^n}
\sum_{i=1}^N \phi\left(\frac{z - Z_i}\ep\right)\\
& \leq & \frac{\| \phi \|_\infty}{N\,\ep^n}
 \# \{ i \text{ s.t. } |z - Z_i|_\infty \leq {\textstyle \frac{L \ep}2} \} \\
\| f_N^\zz \|_\infty & \leq &  \frac{\| \phi \|_\infty}{N\,\ep^n}
 \sup_{z \in \R^n} \# \{ i \text{ s.t. } |z - Z_i|_\infty \leq {\textstyle
\frac{L \ep}2} \},
\end{eqnarray*}
where $\#$ stands for the cardinal (of a finite set). It remains to bound the
supremum on all the cardinals. The first step will be to replace the sup on all
the $z \in \R^n$ by a supremum on a finite number of points. For this, we cover
$[-R^0,R^0]^n$ by $M$ cubes $C_k$ of size $L \ep$, centered at the points
$(c_k)_{k \leq M}$. The number $M$ of squares needed depends on $N$ via $\ep$,
and is bounded by
$$
M \leq \left[ \frac{2(R^0+1)}{L \ep}\right]^n.
$$
Next, for any $z \in \R^d$, there exists a $k \leq M$ such that $|z - c_k| \leq
\frac{L\ep}2$. This implies that
$$
\sup_{z \in \R^n} \# \{ i \text{ s.t. } |z - Z_i|_\infty \leq {\textstyle
\frac{L \ep}2} \} \leq 
\sup_{k \leq M} \# \{ i \text{ s.t. } |c_k - Z_i|_\infty \leq L \ep \}
$$
Now we denote by $H_k^N := \# \{ i \text{ s.t. } |c_k - Z_i|_\infty \leq L \ep
\}$. $H^N_k$ follows a binomial law $B(N,p_k)$ with $ p_k = \int_{2 C_k} f(z)
\,dz $, where $2 C_k$ denotes the square with center $c_k$, but
size $2 L \ep$. Remark that 
$$
p_k \leq \bar p := (2 L \ep)^n \| f\|_\infty. 
$$ 
For any $\lambda$, the exponential moments of $H^N_k$ are therefore given and
bounded by
\begin{eqnarray*}
\E(e^{\lambda H^N_k})  &=& \left[ 1 + (e^{\lambda} -1) p_k \right]^N \\
& \leq & \left[ 1 + (e^{\lambda} -1) (2 L \ep)^n \| f\|_\infty \right]^N  \\
& \leq & e^{(e^{\lambda} -1) N (2 L \ep)^n \| f\|_\infty}.
\end{eqnarray*}
Now for the supremum of the $H^N_k$
\begin{eqnarray*}
\E(e^{\lambda \sup_k H^N_k})  & \leq & \E(e^{\lambda H^N_1}) + \cdots +
(e^{\lambda H^N_M}) \\
& \leq & M e^{(e^{\lambda} -1) N (2 L \ep)^n \| f\|_\infty} \\
& \leq &  \left[ \frac{2(R^0+1)}{L \ep}\right]^n e^{(e^{\lambda} -1) N (2 L
\ep)^n \| f\|_\infty}
\end{eqnarray*}
Using finally Chebyshev's inequality, we get for any $\beta > 0 $
\begin{eqnarray*}
 \PP \left(\|f_N^\zz\|_\infty \geq \beta (2L)^n \|
\phi\|_\infty \| f\|_\infty \right) & \leq &
\PP\left( \sup_k H^N_k \geq \beta  \|f\|_\infty N (2L\ep)^n
\right) \\
& \leq & \E(e^{\lambda \sup_k H^N_k}) e^{- \lambda \beta  \|f\|_\infty N
(2L\ep)^n} \\
& \leq & \left[ \frac{2(R^0+1)}{L \ep}\right]^n e^{ \left( e^{\lambda} -1   -
\lambda \beta \right) N  (2L\ep)^n \| f\|_\infty}.
\end{eqnarray*}
For $\beta >1$, the optimal $\lambda$ is $\ln \beta$ and we get with $c_\phi =
(2L)^n \| \phi\|_\infty$
\[
 \PP \left(\|f_N^\zz\|_\infty \geq \beta c_\phi  \| f\|_\infty \right)
 \leq  \left[ \frac{2(R^0+1)}{L \ep}\right]^n e^{ - \left( \beta \ln \beta -
\beta +1  \right) N (2L\ep)^n \| f\|_\infty}.
\]
With the scaling $\ep(N) = N^{- \frac \gamma n}$, we get
\[
 \PP \left(\|f_N^\zz\|_\infty \geq \beta c_\phi  \| f\|_\infty \right)
 \leq c_0 N^{\gamma}  e^{ - \left( \beta \ln \beta -
\beta +1  \right)  (2L)^n \| f\|_\infty N^{1 - \gamma}}.
\]
Remark finally that the choice of scale $\ep(N) = (\ln N)N^{- \frac 1 n}$
is also sufficient to get a probability vanishing faster than any inverse power.
\end{proof}
\subsection{Existence of strong solutions to Equation \eqref{eq:vlasov}}

This subsection is devoted to the proof of lemma~\ref{lem:KeyEstim}.

\begin{proof}[Proof of the lemma~\ref{lem:KeyEstim}.] 
Given the estimate on $f$, $\rho$ also
belongs to $L^\infty$ with the bound
\[
\|\rho(t,.)\|_{L^\infty(\R^{d})}\leq
C\,K(t)^d\,\|f(t,.,.)\|_{L^\infty(\R^{2d})}.
\]
As we have \eqref{eq:Calpha} with $\alpha<d-1$, $E=F\star_x\rho$ is
Lipschitz. Therefore the solution to \eqref{eq:vlasov} is
given by the characteristics. Namely, we define $X$ and $V$ the unique
solutions to
\[\begin{split}
&\partial_t X(t,s,x,v)=V(t,s,x,v),\quad \partial_t
V(t,s,x,v)=E(t,X(t,s,x,v)),\\
&X(s,s,x,v)=x,\quad V(s,s,x,v)=v.
\end{split}\]
The solution $f$ is now given by
\[
f(t,x,v)=f(0,X(0,t,x,v),V(0,t,x,v)),
\]
with the consequence that
\[
R(t)\leq R(0)+\int_0^t K(s)\,ds,\quad K(t)\leq K(0)+\int_0^t
\|E(s,.)\|_{L^\infty}\,ds.
\]
Then
\[
\|E\|_{L^\infty}\leq
\|\rho\|_{L^1}^{1-\alpha/d}\;\|\rho\|_{L^\infty}^{\alpha/d}, 
\]
which leads to the required inequality. To conclude it is enough to notice that
the $L^1$ and $L^\infty$ norms of $f$
are preserved in this case.
\end{proof}

%

\def\cprime{$'$}

\end{document}